\documentclass[a4paper]{article}

\usepackage{amsmath}
\usepackage{amsthm}
\usepackage{amsxtra}
\usepackage{amsfonts,amssymb}

\usepackage{mathrsfs}

\newtheorem{Th}{Theorem}
\newtheorem{Prop}[Th]{Proposition}
\newtheorem{Lm}[Th]{Lemma}
\newtheorem{Co}[Th]{Corollary}

\theoremstyle{definition}

\newtheorem{Rem}{Remark}
\date{}

\begin{title}
{An analogue of the Schur-Weyl duality for the automorphisms group of a ${\rm II}_1$-factor}
\end{title}
\author{N. I. Nessonov\footnote{B.Verkin Institute for Low Temperature Physics and Engineering
of the National Academy of Sciences of Ukraine, n.nessonov@gmail.com}, \fbox{S. D. Sinel'shchikov}}

\begin{document}
\maketitle

\begin{abstract}
An analogue of the Schur-Weyl duality for the group of automorphisms of the approximately finite dimensional (AFD) ${\rm II}_1$-factor is produced.

\medskip
		
\end{abstract}

\section{Introduction}

Let $M$ be a ${\rm II}_1$-factor with the separable predual $M_*$ and ${\rm tr}$ a unique normal trace
on $M$ such that ${\rm tr}(I)=1$. The inner product $\left<a,b \right>={\rm tr}(b^*a)$ makes $M$ a pre-Hilbert space. Denote by $L^2\left( M,{\rm tr} \right)$ its completion. Let ${\rm
Aut}\,M$ be the automorphism group of $M$ and $U(M)$ the unitary subgroup of $M$. Every $u\in U(M)$
determines the {\it inner} automorphism ${\rm Ad}\,u$ of $M$, ${\rm Ad}\,u(x)=uxu^*$. Denote by ${\rm Inn}\,M$
the subgroup of ${\rm Aut}\,M$ formed by inner automorphisms.

One has a natural unitary representation $\mathfrak{N}$ of
${\rm Aut}\,M$ on the dense subspace $M$ of $L^2\left( M,{\rm tr} \right)$ given by
\begin{equation*}
\mathfrak{N}(\theta)x=\theta(x),\qquad\theta\in{\rm Aut}\,M,\quad x\in M,
\end{equation*}
which is certainly extendable to a representation on $L^2\left(M,{\rm tr}\right)$. Denote by $\mathfrak{N}_I$ the restriction of $\mathfrak{N}$ to the subgroup ${\rm Inn}\,M$.

${\rm Aut}\,M$, being embedded as above into the algebra of bounded operators in $L^2(M,{\rm tr})$, becomes a topological group under the strong operator topology. The subspace $L_0=\left\{v\in L^2( M,{\rm tr}):{\rm tr}(v)=0  \right\}$ is $\mathfrak{N}$-invariant:
$\mathfrak{N}(\theta)L_0=L_0$ for all $\theta\in{\rm Aut}\,M$.

\begin{Th}\label{natural_action}
The restriction $\mathfrak{N}_I^0$ of the representation $\mathfrak{N}_I$ to the invariant subspace $L_0$ is irreducible.
\end{Th}

With an arbitrary ${\rm II}_1$-factor $M$ being replaced in the above settings by the algebra of complex $n\times n$ matrices, Theorem \ref{natural_action} reduces to the well known fact of classical representation theory (see \cite{Kir}, Ch. 3, \S 17.2, Theorem 2). Thus, in case of the approximately finite dimensional (AFD or hyperfinite)
factor $M$, an argument based on approximation of ${\rm II}_1$-factor $M$
by finite dimensional factors is going to be applicable in proving Theorem \ref{natural_action}. However, this theorem in its utmost generality requires a new approach.

Define a diagonal action $\mathfrak{N}^{\otimes k}$ of ${\rm Aut}\,M$ on $L^2(M,\operatorname{tr})^{\otimes k}=L^2\left(M^{\otimes k},\operatorname{tr}^{\otimes k}\right)$ by
\begin{eqnarray*}
\mathfrak{N}^{\otimes k}(\theta)\left(v_1\otimes v_2\otimes\cdots\otimes v_k  \right)=\left(
\mathfrak{N}(\theta)v_1 \right)\otimes\left( \mathfrak{N}(\theta)v_2 \right)\otimes\cdots\otimes\left(
\mathfrak{N}(\theta)v_k \right).
\end{eqnarray*}
Additionally, the symmetric group $\mathfrak{S}_k$ acts on $L^2\left(M^{\otimes k},\operatorname{tr}^{\otimes k}\right)$ by permutations
\begin{eqnarray}\label{action_of+symmetric_group}
\;^k\!\mathcal{P}(s)\left(v_1\otimes v_2\otimes\cdots\otimes v_k  \right)=v_{s^{-1}(1)}\otimes
v_{s^{-1}(2)}\otimes\cdots\otimes v_{s^{-1}(k)}.
\end{eqnarray}
Since the operators $\mathfrak{N}^{\otimes k}(\theta)$ and $\;^k\!\mathcal{P}(s)$ commute, we obtain a
representation $\mathcal{F}$ of the group ${\rm Aut}\,M\times\mathfrak{S}_k$, $\mathcal{F}\left( \theta,s
\right)=\mathfrak{N}^{\otimes k}(\theta)\cdot \;^k\!\mathcal{P}(s)$.

Denote by $\mathfrak{N}^{\otimes k}_0$ and $\;^k\!\mathcal{P}_0$ the restrictions of the representations $\mathfrak{N}^{\otimes k}$ and $\;^k\!\mathcal{P}$ to the subspace  $L_0^{\otimes k}\subset L^2(M,\operatorname{tr})^{\otimes k}$.

Recall that the irreducible representations of $\mathfrak{S}_k$ are parameterized by the unordered
partitions of $k$.  Denote the set of all such partitions by $\Upsilon_k$. Let
$\lambda\in\Upsilon_k$ and let $\chi_\lambda$ be the character of the corresponding irreducible
representation ${\rm R}_\lambda$. Denote by ${\rm dim} \lambda$ the dimension of ${\rm R}_\lambda$. The operator
\begin{eqnarray}
P^\lambda=\frac{{\rm dim} \lambda }{k!}\sum\limits_{s\in\mathfrak{S}_k}\chi_\lambda(s)\;^k\!\mathcal{P}(s)
\end{eqnarray}
is an orthogonal projection in the centre of the $w^*$-algebra generated by the operators $\left\{
\mathcal{F}\left( \theta,s \right)  \right\}_{(\theta,s)\in{\rm Aut}\,M\times\mathfrak{S}_k}$.
Denote by $\mathcal{F}^\lambda_0$ the representation $\mathcal{F}$ restricted to the subspace
$H^\lambda_0=P^\lambda\left( L_0^{\otimes k} \right)$.

\begin{Th}\label{main_Schur-Weyl}
Let $M$ be an AFD ${\rm II}_1$-factor. Then the commutant of the set $\mathfrak{N}^{\otimes k}_0({\rm Aut}\,M)$ is generated by $\;^k\!\mathcal{P}_0(\mathfrak{S}_k)$.
\end{Th}
\begin{Co}
  The representation $\mathcal{F}_0^\lambda$ of ${\rm Aut}\,M\times\mathfrak{S}_k$ is
irreducible. With different $\lambda,\zeta\in\Upsilon_k$, the restrictions of $\mathcal{F}_0^\lambda$ and $\mathcal{F}_0^\zeta$ to the subgroup ${\rm Aut}\, M$ are not quasi-equivalent.
\end{Co}
Representation $\;^k\!\mathcal{P}$ can be extended to a representation $\;^k\!\mathcal{P}^{\mathscr{I}_k}$ of {\it the symmetric inverse semigroup} $\mathscr{I}_k$, which can realize  as a semigroup of $\{0,1\}$-matrices $a=\left[  a_{ij} \right]_{i,j=1}^k$ with the ordinary matrix multiplication  in such a way that $a$ has at most one nonzero entry in each row and each column. We denote by $\epsilon_i$ a diagonal matrix $[a_{pq}]$ such that $a_{ii}=0$ and $a_{pq}=\delta_{pq}$, if $p\neq i$ or $q\neq i$. Of course, $\mathfrak{S}_k\subset \mathscr{I}_k$. Define operator $\;^k\!\mathcal{P}^{\mathscr{I}_k}(\epsilon_i)$ on $L^2\left(M^{\otimes k},\operatorname{tr}^{\otimes k}\right)$ as follows
\begin{eqnarray*}
\;^k\!\mathcal{P}^{\mathscr{I}_k}(\epsilon_i)\left(\cdots v_{i-1}\otimes v_i\otimes v_{i+1}\cdots\right)={\rm tr}(v_i)(\epsilon_i)\left(\cdots v_{i-1}\otimes {\rm I}\otimes v_{i+1}\cdots\right).
\end{eqnarray*}
We set $\;^k\!\mathcal{P}^{\mathscr{I}_k}(s)=\;^k\!\mathcal{P}(s)$, if $s\in \mathfrak{S}_k$. Then $\;^k\!\mathcal{P}^{\mathscr{I}_k}$ is extended to a representation of the semigroup $\mathscr{I}_k$. Using Theorem \ref{main_Schur-Weyl}, we prove in section \ref{finite_inverse_semigroup} next statement.
\begin{Th}
If $M$ is an AFD ${\rm II}_1$-factor then the commutant of $\mathfrak{N}^{\otimes k}({\rm Aut}\,M)$ is generated by $\;^k\!\mathcal{P}^{\mathscr{I}_k}\left( \mathscr{I}_k\right)$.
\end{Th}
Using the embedding
$$L^2\left( M,{\rm tr} \right)^{\otimes n}\ni m_1\otimes\ldots\otimes m_n\mapsto m_1\otimes\ldots\otimes m_n\otimes{\rm I}\in L^2\left( M,{\rm tr} \right)^{\otimes (n+1)},$$
we identify $~L^2\left( M,{\rm tr} \right)^{\otimes n}$ with the subspace in $~L^2\left( M,{\rm tr} \right)^{\otimes (n+1)}$.
  Denote by $L^2\left( M,{\rm tr} \right)^{\otimes \infty}$ the completion of the pre-Hilbert space $\bigcup\limits_{n=1}^\infty L^2\left( M,{\rm tr} \right)^{\otimes n}$. It is convenient to consider $\bigcup\limits_{n=1}^\infty L^2\left( M,{\rm tr} \right)^{\otimes n}$ as the linear span of the vectors $$v_1\otimes \cdots\otimes v_n\otimes{\rm I}\otimes{\rm I}\otimes\cdots,~\text{ where}~ v_j\in M.$$ At the same time, we will to identify  $L^2\left( M,{\rm tr} \right)^{\otimes n}$ with the closure of the linear span of all vectors $v_1\otimes \cdots\otimes v_n\otimes v_{n+1}\otimes\cdots$, where $v_i={\rm I}$ for all $i>n$.
 Define the representation $\mathfrak{N}^{\otimes \infty}$ of group ${\rm Aut}\, M$ as follows
\begin{eqnarray*}
\mathfrak{N}^{\otimes \infty}(\theta)\left(v_1\otimes\cdots\otimes v_n\otimes \cdots  \right)=\left(
\mathfrak{N}(\theta)v_1 \right)\otimes\cdots\otimes\left(
\mathfrak{N}(\theta)v_n \right)\otimes \cdots.
 \end{eqnarray*}
 The infinite symmetric group $\mathfrak{S}_\infty$ acts on  $L^2\left( M,{\rm tr} \right)^{\otimes \infty}$ by permutations
 \begin{eqnarray*}
\;^\infty\!\mathcal{P}(s)\left(v_1\otimes \cdots\otimes v_n\otimes\cdots  \right)=v_{s^{-1}(1)}\otimes
\cdots\otimes v_{s^{-1}(n)}\otimes\cdots, \;\;\; s\in \mathfrak{S}_\infty.
 \end{eqnarray*}
 We prove in section \ref{infinite_Sch_Weyl} the following statement.
 \begin{Th}
If $M$ is an AFD ${\rm II}_1$-factor then the commutant of  $\mathfrak{N}^{\otimes \infty}({\rm Aut}\,M)$ is generated by $\;^k\!\mathcal{P}\left(\mathfrak{S}_\infty\right)$.
\end{Th}

\section{Proof of Theorem \ref{natural_action}}\label{pth1}

Let $M$ be a ${\rm II}_1$-factor. Denote by $B\left(L^2(M,\operatorname{tr})\right)$ the algebra of all bounded operators on $L^2(M,\operatorname{tr})$. Recall that a $w^*$-subalgebra $\mathfrak{A}\subset M $ is called {\it masa} (maximal Abelian subalgebra) if $(\mathfrak{A}'\cap M)=\mathfrak{A}$, where $$\mathfrak{A}'=\left\{\left.b\in B\left(L^2(M,\operatorname{tr})\right)\right|\:ba=ab\text{\ for all\ }a\in\mathfrak{A}\right\}$$ is the commutant of $\mathfrak{A}$. Let  $\mathcal{N}(\mathfrak{A})=\left\{ u\in U(M):u\mathfrak{A}
u^*= u^*\mathfrak{A} u =\mathfrak{A}\right\}$ be the {\it normalizer} of $\mathfrak{A}$. Let $\mathcal{N}(\mathfrak{A})^{\prime\prime}$ be the
$w^*$-subalgebra generated by $\mathcal{N}(\mathfrak{A})$. A masa $\mathfrak{A}$ is said to be Cartan if
$\mathcal{N}(\mathfrak{A})^{\prime\prime}=M$.

We need the following claim from \cite{Sinclair} (p. 242).

\begin{Prop}\label{afm}
There exists a masa $\mathfrak{A}$ in $M$ and an AFD-subfactor $F$ of $M$ containing $\mathfrak{A}$ such that
$\mathfrak{A}$ is a Cartan subalgebra of $M$ and $F^\prime\cap M=\mathbb{C}I$.
\end{Prop}

It is well known that, in the context of latter Proposition, one can readily find the family $\left\{ K_n  \right\}_{n=1}^\infty$ of pairwise commuting
${\rm I}_2$-subfactors $K_n\subset F$ which generate $F$. Fix a system of matrix units $\left\{\,^n\!e_{ij}  \right\}_{i,j=1}^2\subset K_n$. Denote by $\mathfrak{A}_K$ an Abelian $w^*$-subalgebra generated by $\left\{\,^r\!e_{11},\,^r\!e_{22}\right\}_{r=1}^\infty$.  It is easy to check  that $\mathfrak{A}_K$ is a Cartan subalgebra in $F$. Since any two Cartan masas $\mathfrak{A}_1$ and $\mathfrak{A}_2$ of $F$ are conjugate, i. e. there exists $\theta\in\operatorname{Aut}F$ such that $\theta\left(\mathfrak{A}_1\right)=\mathfrak{A}_2$, we can assume without loss of generality that the masa $\mathfrak{A}$ coincides with $\mathfrak{A}_K$.

Let $E$ be a unique {\it conditional expectation} of $M$ onto $\mathfrak{A}$ with respect to ${\rm tr}$
\cite{TAKES1}. In particular, $E$ is the orthogonal projection of the subspace $L_0$ onto the subspace
$$
L_0^\mathfrak{A}=\left\{x\in L^2\left(\mathfrak{A},\mathrm{tr}\right):\:\mathrm{tr}(x)=0\right\}.
$$

We claim that $E$ belongs to the $w^*$-algebra generated by $\mathfrak{N}({\rm Aut}\,M)$. To see this, consider a
family $\left\{ \Gamma_n  \right\}$ of Abelian finite subgroups of ${\rm Aut}\,M$. Namely,
$\Gamma_n$ is generated by the inner automorphisms ${\rm Ad}\,u$, with the unitaries $u$ belonging to the collection $\left\{\,^r\!e_{11}-\,^r\!e_{22}\right\}_{r=1}^n$. Since $\mathfrak{A}$ is a masa in $M$, one has, in view of Proposition \ref{afm}, that
\begin{equation}\label{masa}
\left(\left\{\,^r\!e_{11}-\,^r\!e_{22}\right\}_{r=1}^\infty\right)'=
\mathfrak{A}.
\end{equation}
Denote by $E_n$ the orthogonal projection in $L^2(M,{\rm tr})$ determined by its values on the dense subset $M\subset L^2(M,{\rm tr})$
\begin{eqnarray}\label{conditional_Exp}
M\ni x\stackrel{E_n}{\mapsto}|\Gamma_n|^{-1}\sum\limits_{\gamma\in\Gamma_n}\gamma(x).
\end{eqnarray}
Since $E_r\geq E_{r+1}$, the sequence $E_r$ converges in the strong operator topology. Let $\lim\limits_{r\to\infty}E_r=\widetilde{E}$. Hence, an application of (\ref{masa}) and (\ref{conditional_Exp}) yields
\begin{equation*}
\begin{split}
&\widetilde{E}(x)\in \mathfrak{A},
\\ &{\rm tr}(\widetilde{E}(x))={\rm tr}(x)\qquad\text{\ for all\ } x\in M,
\\ &\widetilde{E}(axb)=a\widetilde{E}(x)b\qquad\text{\ for all\ } a,b\in\mathfrak{A},\quad x\in M.
\end{split}
\end{equation*}
Therefore, $\widetilde{E}$ is the conditional expectation onto $\mathfrak{A}$. It follows that
$\widetilde{E}=E$. Thus, in view of (\ref{conditional_Exp}), $E$ belongs to the $w^*$-algebra generated by
$\mathfrak{N}({\rm Inn}\,M)$. Therefore,
\begin{eqnarray}
A^\prime L_0^\mathfrak{A}\subset L_0^\mathfrak{A} \text{\ for all\ } A^\prime \in\left(
\mathfrak{N}_I^0({\rm Inn}\,M) \right)^\prime.
\end{eqnarray}
The uniqueness of conditional expectation implies
\begin{eqnarray*}
{\rm Ad}\,u\;\circ E\circ{\rm Ad}\,u^*\;=E \text{\ for all\ } u\in\mathcal{N}(\mathfrak{A}).
\end{eqnarray*}
This is to be rephrased by claiming that the action of ${\rm Ad}\,\mathcal{N}(\mathfrak{A})$ leaves invariant $L_0^\mathfrak{A}$:
\begin{eqnarray}
{\rm Ad}\,u \;(a)\in L_0^\mathfrak{A} \text{\ for all\ } a\in L_0^\mathfrak{A},\quad u\in\mathcal{N}(\mathfrak{A}).
\end{eqnarray}
Now to prove Theorem \ref{natural_action}, it suffices to demonstrate the following:
\begin{description}
\item[\bf a)] the action of $\mathcal{N}(\mathfrak{A})$, $u\mapsto{\rm Ad}\,u$, leaves no non-trivial closed subspace of $L_0^\mathfrak{A}$ invariant;

\item [\bf b)] the subspace $ L_0^\mathfrak{A}\subset L_0$ is cyclic with respect to $\mathfrak{N}({\rm Inn}\,M)$; i. e. the smallest closed subspace, containing $\bigcup\limits_{\theta\in
      {\rm Inn}\,M}\mathfrak{N}(\theta)L_0^\mathfrak{A} $, is just $L_0$.
\end{description}

Let us start with proving {\bf a)}. Consider an arbitrary unitary $$u\in \left\{ K_1,K_2,\ldots, K_n  \right\}^{\prime\prime},$$ to be expanded as
$$
u=\sum\limits_{j_1,k_1,j_2,k_2,\ldots,j_n,k_n=1}^2
u_{j_1k_1\,j_2k_2\,\ldots\,j_nk_n}\;\,^1\!e_{j_1k_1}\,^2\!e_{j_2k_2}\,\ldots\,^n\!e_{j_nk_n},
$$
where $u_{j_1k_1\,j_2k_2\,\ldots\,j_nk_n}\in \mathbb{C}$. Denote by $\mathfrak{S}_{2^n}$ the group of all bijections of the set $X_n=\left\{\left( i_1,i_2,\ldots,i_n \right),\;\;i_r\in\{1,2\}\right\}$. Within our current argument, the symmetric group $\mathfrak{S}_{2^n}$ is about to be identified with the subgroup
$$
\left\{u\in\left\{ K_1,K_2,\ldots, K_n  \right\}''\cap
U(M):u_{j_1k_1\,j_2k_2\,\ldots\,j_nk_n}\in\{0,1\} \right\}\subset \mathcal{N}(\mathfrak{A}),
$$
in terms of the above expansion for $u\in\left\{K_1,K_2,\ldots,K_n\right\}''$.
It is also convenient to denote by $\mathbf{i}_n$ the multiindex $\left(i_1,i_2,\ldots,i_n \right)$. Clearly, the collection of vectors $\left\{\mathbf{\mathfrak{e}}_{\mathbf{i}_n}=
\,^1\!e_{i_1i_1}\,^2\!e_{i_2i_2}\,\ldots\,^n\!e_{i_ni_n}\right\}$ forms an orthogonal basis of the subspace $\mathfrak{A}_n=\mathfrak{A}\cap\left\{K_1,K_2,\ldots,K_n\right\}''$.

Let $\mathfrak{E}_n$ be the orthogonal projection of $L^2\left( \mathfrak{A},{\rm tr} \right)$ onto
 $\mathfrak{A}_n$, and consider a bounded operator $B^\prime\in\left( {\rm Ad}\, \mathcal{N}\left(
 \mathfrak{A} \right) \right)^{\prime}$. It is clear that  $\,^n\!B^\prime\stackrel{\operatorname{def}}{=}\mathfrak{E}_nB^\prime
 \mathfrak{E}_n$ belongs to $\left( {\rm Ad}\,\mathfrak{S}_{2^n} \right)^\prime$ and
 \begin{eqnarray}\label{approximation}
 \lim\limits_{n\to\infty}\,^n\!B^\prime=\,B^\prime \text{\ in the strong operator topology}.
 \end{eqnarray}
Hence, denoting the matrix element $\left(\,^n\!B^\prime
\mathbf{\mathfrak{e}}_{\mathbf{i}_n},\mathbf{\mathfrak{e}}_{\mathbf{j}_n} \right)$ by
$\,^n\!B^\prime_{\mathbf{i}_n\,\mathbf{j}_n}$, one has
\begin{eqnarray*}
\,^n\!B^\prime_{s(\mathbf{i}_n)\,s(\mathbf{j}_n)}=\,^n\!B^\prime_{\mathbf{i}_n\,\mathbf{j}_n}\text{ for
all } s\in\mathfrak{S}_{2^n}.
\end{eqnarray*}
Therefore, there exist $\gamma,\delta\in \mathbb{C}$ such that
 \begin{eqnarray*}
 \,^n\!B^\prime_{\mathbf{i}_n\,\mathbf{j}_n}=\left\{
 \begin{array}{rl}
 \gamma,&\text{ if\ } \mathbf{i}_n\neq\mathbf{j}_n;\\
 \delta,&\text{ if\ }  \mathbf{i}_n=\mathbf{j}_n.
 \end{array}\right.
 \end{eqnarray*}
It follows that
 \begin{eqnarray*}
 \,^n\!B^\prime\eta=(\delta-\gamma)\eta \text{ for all } \eta\in L_0^\mathfrak{A}\cap\mathfrak{A}_n.
 \end{eqnarray*}
 Hence, applying (\ref{approximation}), we obtain that $B^\prime\eta=(\delta-\gamma)\eta \text{ for all }
 \eta\in L_0^\mathfrak{A}$. This proves {\bf a}).

Turn to proving {\bf b)}. It suffices to demonstrate that, given a self-adjoint $B\in M$ and $\epsilon>0$, there
exist $A\in\mathfrak{A}$ and $U\in U(M)$ with the property
 \begin{eqnarray}\label{inequality}
 \|B-UAU^*\|<\epsilon, \text{\ where\ } \|\cdot\| \text{\ stands for the operator norm}.
 \end{eqnarray}
Choose a positive integer $n>\frac{\|B\|}{\epsilon}$ and consider the set of reals
$$
\Delta_l=\left\{r\left|\:\frac{2(l-1)\|B\|}{n}-\|B\|<r
\le\frac{2l\|B\|}{n}-\|B\|\right.\right\}
$$
for each $l=0,1,\ldots,n$. Let $E(\Delta_l)$ be the associated spectral projection related to the spectral decomposition of $B$. Under this setting, with
$$
\alpha_l=\frac{(2l-1)\|B\|}{n}-\|B\|,\qquad B_n=\sum\limits_{l=0}^n\alpha_lE\left(\Delta_l\right),
$$
we conclude that
\begin{equation}
\|B-B_n\|\leq\epsilon.
\end{equation}
One can readily find a family $\left(F_l\right)_{l=0}^n$ of pairwise orthogonal projections in $\mathfrak{A}$ such that $\mathrm{tr}\left(F_l\right)=\mathrm{tr}\left(E(\Delta_l)\right)$. Thus we can also select partial isometries $u_l\in M$ with the properties $u_lu_l^*=E(\Delta_l)$ and $u_l^*u_l=F_l$ for all $l=1,2,\ldots,n$. It follows that $U=\sum\limits_{l=0}^nu_l$ is a unitary operator, and with
$A=\sum\limits_{l=0}^n\alpha_lF_l$ the inequality (\ref{inequality}) holds.

\section{Proof of theorem \ref{main_Schur-Weyl}}\label{pr_th_2}

Notice first that there exists a family $\left\{ N_j \right\}_{j=1}^\infty$ of pairwise commuting type
$\mathrm{I}_k$ subfactors $N_j\subset M$ generating $M$. Let $M_{jJ}=\left(\left\{ N_l
\right\}_{l=j}^{J}\right)^{\prime\prime}$. Fix a system of matrix units $\left\{ \,^n\!e_{ij}
\right\}_{i,j=1}^k\subset N_n$. Denote by $\mathfrak{A}$ an Abelian $w^*$-subalgebra generated by $\left\{
\,^l\!e_{11}, \,^l\!e_{22}, \,\ldots, \,^l\!e_{kk} \right\}_{l=1}^\infty$. One can reproduce here the argument used at the beginning of Section \ref{pth1} to demonstrate that $\mathfrak{A}$ is a Cartan MASA in $M$.

\subsection{The conditional expectation from $M^{\otimes k}$ onto $\mathfrak{A}^{\otimes k}$}\label{CE}

It is well known that there exists a unique conditional expectation $^k\!{\mathop{E}}$ from the ${\rm II}_1$-factor
$M^{\otimes k}$ onto the Cartan MASA $\mathfrak{A}^{\otimes k}\subset M^{\otimes k}$. Recall that $\,^k\!{\mathop{E}}$ is
uniquely determined by the following properties (see \cite{TAKES1}):
\begin{description}
\item[\bf 1)] $\,^k\!E$ is continuous with respect to the strong operator topology and $\,^k\!E\,{\rm I}={\rm I}$;

\item[\bf 2)] $\,^k\!E\left( a_1ma_2 \right)=a_1\,^k\!E(m)a_2$ for all $m\in M^{\otimes k}$ and
      $a_1,a_2\in\mathfrak{A}^{\otimes k}$;

\item[\bf 3)] ${\rm tr}^{\otimes k}(\,^k\!Em)={\rm tr}^{\otimes k}(m)$ for all $m\in M^{\otimes k}$.
\end{description}
We prove be\-low that $\,^k\!E$ be\-longs to  $\left(\mathfrak{N}^{\otimes k}\left(\operatorname{Ad}\,U(M)\right)\right)^{''}$.

With $\mathbf{i}_J=(i_1,i_2,\ldots,i_J)$, let $\mathfrak{e}_{\mathbf{i}_J}$ stand for the minimal projection
$$\,^1\!e_{i_1i_1}\,^2\!e_{i_2i_2}\cdots\,^J\!e_{i_Ji_J}$$ of the algebra  $M_{1J}\cap\mathfrak{A}$. Let
$\,^n\!f$ be the embedding of the finite set
 $$\mathfrak{I}_J=\left\{ \mathbf{i}_J=(i_1,i_2,\ldots,i_J)  \right\}_{i_1,i_2,\ldots,i_J=1}^k$$
into $\left\{ n+1,n+2,\ldots  \right\}$. Set
$\,^p\!u=\,^p\!e_{k1}+\sum\limits_{l=1}^{k-1}\,^p\!e_{l\;l+1}\;\in N_p$.

\begin{Lm}\label{approximation_CE}
Consider the unitary $\,^J\!U_n=\sum\limits_{\mathbf{i}_J\in\mathfrak{I}_J}\;\mathfrak{e}_{\mathbf{i}_J}\cdot
\,^p\!u$, where $p=\,^n\!f\left( \mathbf{i}_J \right)$ and $n>J$. Then for any $ m\in M$ the sequence $\mathfrak{N}\left(\operatorname{Ad}\left(\sideset{^J}{_n}{\mathop{U}}
\right)\right)m$ converges in the weak operator topology so that
$\lim\limits_{n\to\infty}\mathfrak{N}\left(\operatorname{Ad}(^J\!U_n)\right)m
=E_J(m)$, with
\begin{equation}\label{formula_approx_CE}
E_J(m)=\sum\limits_{\mathbf{i}_J\in\mathfrak{I}_J}\;\;\mathfrak{e}_{\mathbf{i}_J}\cdot m\cdot\;\mathfrak{e}_{\mathbf{i}_J}\in\mathfrak{A}'\cap M_{1J}.
\end{equation}
In particular, $E_J$ belongs to the $w^*$-algebra generated by $\mathfrak{N}\left(\operatorname{Ad}\,U(M)\right)$.
\end{Lm}

\begin{proof}
Since the algebra $\bigcup\limits_{Q=1}^\infty \; M_{1Q}$ is dense in $M$ in the strong operator topology, one
can assume without loss of generality that $m\in M_{1L}$, where $L>J$. Under this assumption, we have with
$n>L$
\begin{eqnarray*}
\,^J\!U_n\cdot\,m \cdot
\,^J\!U_n^*=\sum\limits_{\mathbf{i}_J,\mathbf{r}_J\in\mathfrak{I}_J}\;\mathfrak{e}_{\mathbf{i}_J}\cdot
m\cdot \;\mathfrak{e}_{\mathbf{r}_J}\cdot \,^p\!u\cdot \,^q\!u^*,
\end{eqnarray*}
where $p=\,^n\!f\left(
\mathbf{i}_J \right)$, $q= \,^n\!f\left( \mathbf{r}_J \right)$.
Note that with $\mathbf{i}_J\neq\mathbf{r}_J$ one has
$$\lim\limits_{n\to\infty}\,^p\!u\cdot \,^q\!u^*={\rm tr}\left(\,^p\!u\cdot \,^q\!u^*  \right)\operatorname{I}=0$$
in the weak operator topology. Therefore, $\lim\limits_{n\to\infty}\,^J\!U_n\cdot\,m \cdot \,^J\!U_n^*=E_J(m)$.
\end{proof}

\begin{Rem}\label{Remark_conditional_ex}
Clearly, $E_J$ is an orthogonal projection in $L^2\left( M,{\rm tr} \right)$.
Also, one readily observes that $E_J\geq E_{J+1}$ for all $J$. Hence for any $m\in L^2(M,{\rm tr})$ there exists $$\lim\limits_{J\to\infty} E_J(m)=E(m).$$ In particular,
\begin{equation}\label{Cond_expect_local}
E(m)=E_J(m) \text{\ for all\ } m\in M_{1J}.
\end{equation}
It is easy to verify that $E$ is the unique {\it conditional expectation} of $M$ onto $\mathfrak{A}$ with respect to ${\rm tr}$ \cite{TAKES1}. On the other hand, {\bf 1}) -- {\bf 3}) are valid also for the projection $E^{\otimes k}$. The uniqueness of conditional expectation now implies
\begin{equation}\label{Cond_expect_tensor_product}
\,^k\!E\left( m_1\otimes m_2\otimes\cdots\otimes m_k \right)= E(m_1)\otimes  E(m_2)\otimes\cdots\otimes E( m_k)
\end{equation}
for all  $m_1, m_2,\ldots, m_k\in M$.
\end{Rem}

\begin{Prop}\label{limits_of_E_J}
$\,^k\!E\in\left(\mathfrak{N}^{\otimes k}\left(\operatorname{Ad}\,U(M)\right)\right)^{\prime\prime}$.
\end{Prop}

\begin{proof}
Let $E_J^{\otimes k}(m_1\otimes m_2\otimes\cdots\otimes m_k)\stackrel{\mathrm{def}}{=}E_J(m_1)\otimes E_J(m_2)\otimes\cdots\otimes
E_J(m_k)$. By Lemma \ref{approximation_CE},
\begin{eqnarray}
E_J^{\otimes k}\in \left(\mathfrak{N}^{\otimes k}\left(\operatorname{Ad}\,U(M)\,\right)\right)''.
\end{eqnarray}
$E_J^{\otimes k}$ is an orthogonal projection in $L^2\left( M^{\otimes k},{\rm tr}^{\otimes k}
\right)$ and $E_J^{\otimes k}\geq E_L^{\otimes k}$ for all $L>J$. It follows that for any $m\in L^2\left(
M^{\otimes k},{\rm tr}^{\otimes k} \right)$ there exists $\lim\limits_{J\to\infty} E_J^{\otimes
k}(m)\stackrel{\mathrm{def}}{=}\widetilde{E}(m)\in M^{\otimes k}\cap\left(\mathfrak{A}^{\otimes k}\right)'$. Therefore, $\widetilde{E}\in  \left(\mathfrak{N}^{\otimes k}\left({\rm Ad}\, U(M)
\right)\right)^{\prime\prime}$. An application of (\ref{formula_approx_CE}) allows one to verify that {\bf 1}) -- {\bf 3}) are valid for $\widetilde{E}$. Since $\mathfrak{A}^{\otimes k}$ is a MASA in
$M^{\otimes k}$, we conclude that $\widetilde{E}\left(  M^{\otimes k}\right)=\mathfrak{A}^{\otimes k}$. Therefore,
$\widetilde{E}$ is a conditional expectation from  $M^{\otimes k}$ onto $\mathfrak{A}^{\otimes k}$, hence $\widetilde{E}=\,^k\!E=E^{\otimes k}$ by \eqref{Cond_expect_tensor_product}.
\end{proof}

\subsection{The operators $\,^k\!E\cdot\mathfrak{N}^{\otimes k}(u)\,\cdot\,^k\!E$ on $L^2\left( \mathfrak{A}^{\otimes k}, {\rm tr}^{\otimes k} \right)$.}\label{mu_section}

With $\mathbf{i}_J=(i_1,i_2,\ldots,i_J)  $, $\mathbf{i}_J^\prime=(i_1^\prime,i_2^\prime,\ldots,i_J^\prime)$, denote the partial isometry \;\;\;\;\;\;\; $\,^1\!e_{i_1i_1^\prime}\,^2\!e_{i_2i_2^\prime}\cdots\,^J\!e_{i_Ji_J^\prime}$ $\in M_{1J}$ by $\mathfrak{e}_{\mathbf{i}_J\,\mathbf{i}_J^\prime}$. Given a collection $\,^l\!x\in M_{1J}$, $1\le l\le k$, we use below the expansion $$\,^l\!x=\sum\limits_{\mathfrak{i}_J,\mathfrak{i}^\prime_J\in\mathfrak{I}_J}
\,^l\!c_{\mathbf{i}_J\,\mathbf{i}_J^\prime}
\mathfrak{e}_{\mathbf{i}_J\,\mathbf{i}_J^\prime}\in M_{1J},  \text{ where } \,^l\!c_{\mathbf{i}_J\,\mathbf{i}_J^\prime}\in\mathbb{C}.$$
In view of \eqref{Cond_expect_tensor_product} one has
\begin{equation}
\begin{split}
&\,^k\!{\mathop{E}}\left(^1\!x\otimes\,^2\!x \otimes\cdots\otimes\,^k\!x \right)=E_J(^1\!x)\otimes E_J(^2\!x ) \otimes\cdots\otimes E_J(^k\!x)
\\ =&\left(\sum\limits_{\mathfrak{i}_J\in\mathfrak{I}_J}\!\!\!
\sideset{^1}{_{\mathbf{i}_J\mathbf{i}_J}}{\mathop{\!c}}
\mathfrak{e}_{\mathbf{i}_J\,\mathbf{i}_J} \right)\otimes \left(\sum\limits_{\mathfrak{i}_J\in\mathfrak{I}_J}\!\!\!
\sideset{^2}{_{\mathbf{i}_J\mathbf{i}_J}}{\mathop{\!c}}
\mathfrak{e}_{\mathbf{i}_J\,\mathbf{i}_J}\right)\otimes\cdots\otimes
\left(\sum\limits_{\mathfrak{i}_J\in\mathfrak{I}_J}\!\!\!
\sideset{^k}{_{\mathbf{i}_J\mathbf{i}_J}}{\mathop{\!c}}
\mathfrak{e}_{\mathbf{i}_J\,\mathbf{i}_J} \right).
\end{split}
\end{equation}
Note that in Subsection \ref{CE} another notation $\mathfrak{e}_{\mathbf{i}_J}$ was used for $\mathfrak{e}_{\mathbf{i}_J\,\mathbf{i}_J}$.

Consider a unitary $u=\sum\limits_{\mathfrak{i}_J,\mathfrak{i}^\prime_J\in\mathfrak{I}_J}\;
u_{\mathbf{i}_J\,\mathbf{i}_J^\prime}\cdot\mathfrak{e}_{\mathbf{i}_J\,\mathbf{i}_J^\prime}\;\in M_{1J}$ and a collection $\,^l\!a=\sum\limits_{\mathfrak{i}_J\in\mathfrak{I}_J}\;\,^l\!a_{\mathbf{i}_J}\cdot
\mathfrak{e}_{\mathbf{i}_J\,\mathbf{i}_J}\;$ $\in M_{1J}\cap\mathfrak{A}$, $1\le l\le k$, where $u_{\mathbf{i}_J\,\mathbf{i}_J^\prime}, \,^l\!a_{\mathbf{i}_J}\in \mathbb{C}$. Since
\begin{multline*}
\,^k\!E\left(\mathfrak{N}^{\otimes k}\!({\rm Ad}\,u)\!\left( \,^1\!a\otimes\,^2\!a \otimes\cdots\otimes \,^k\!a \right)\right)
\\ =\,^k\!E \left( u\cdot\,^1\!a\cdot\,u^*\otimes \,u\cdot\,^2\!a\cdot\,u^* \otimes\cdots\otimes\,u \cdot\,^k\!a\cdot\,u^* \right),
\end{multline*}
an application of (\ref{Cond_expect_local}) and (\ref{Cond_expect_tensor_product}) yields
\begin{equation}\label{Operator_unistochastic}
\begin{split}
&\,^k\!E\left(\mathfrak{N}^{\otimes k}\!({\rm Ad}\,u)\right)\!\left( \,^1\!a\otimes\,^2\!a \otimes\cdots\otimes \,^k\!a \right)=\,^1\!b\otimes\,^2\!b \otimes\cdots\otimes \,^k\!b,
\text{ where }\\
&\,^l\!b=\sum\limits_{\mathfrak{i}_J\in\mathfrak{I}_J}\;\,^l\!b_{\mathfrak{i}_J}\cdot
\mathfrak{e}_{\mathfrak{i}_J\,\mathfrak{i}_J}\,\in M_{1J}\cap\mathfrak{A}\;\text{ and } \,^l\!b_{\mathfrak{i}_J}=\sum\limits_{\mathfrak{k}_J\in\mathfrak{I}_J}\left|u_{\mathfrak{i}_J\,\mathfrak{k}_J} \right|^2\cdot \,^l\!a_{\mathfrak{k}_J}.\;\;\;\;
\end{split}
\end{equation}
This way the map
$$
\mu:M_{1J}\cap U(M)\to M_{1J};\qquad \sum\limits_{\mathfrak{i}_J,\mathfrak{i}'_J\in\mathfrak{I}_J}\;
u_{\mathbf{i}_J\,\mathbf{i}_J'}\cdot
\mathfrak{e}_{\mathbf{i}_J\,\mathbf{i}_J'}\mapsto
\sum\limits_{\mathfrak{i}_J,\mathfrak{i}'_J\in\mathfrak{I}_J}\;
\left|u_{\mathbf{i}_J\,\mathbf{i}_J'}\right|^2\cdot
\mathfrak{e}_{\mathbf{i}_J\,\mathbf{i}_J'}.
$$
is introduced. It is to be studied and used in what follows.

Note that $\left|u_{\mathfrak{i}_J\mathfrak{k}_J}\right|^2$ form a {\it doubly stochastic} matrix (see Section \ref{U_to_DS}), hence
\begin{eqnarray}
\sum\limits_{\mathfrak{i}_J\in\mathfrak{I}_J}\;\,^l\!a_{\mathbf{i}_J}= \sum\limits_{\mathfrak{i}_J\in\mathfrak{I}_J}\;\,^l\!b_{\mathbf{i}_J}\; \text{\ for all\ }l.
\end{eqnarray}

\subsubsection{Some properties of the map $\mu$}

Set $n=k^J$. To simplify the notation, it is custom (and really convenient) to identify $m=\sum\limits_{\mathfrak{i}_J,\mathfrak{i}^\prime_J\in\mathfrak{I}_J}\;
m_{\mathbf{i}_J\,\mathbf{i}_J^\prime}\cdot\mathfrak{e}_{\mathbf{i}_J\,\mathbf{i}_J^\prime}$ $\in M_{1J}$ with the associated matrice $\left[m_{\mathbf{i}_J\,\mathbf{i}_J'}\right]$. Let $M_{1J}(\mathbb{R})$ be the subset of real matrices in $M_{1J}$. Denote also by $GL(n,\mathbb{R})$
the subgroup of all invertible elements of $M_{1J}(\mathbb{R})$. A matrix $m=\left[m_{\mathbf{i}_J\,\mathbf{i}_J'}\right]\in M_{1J}$
is said to be {\it doubly stochastic} if its elements satisfy
\begin{equation*}
\begin{split}
&m_{\mathbf{i}_J\,\mathbf{i}_J'}\ge 0\text{\ for all\ }\mathbf{i}_J\,\mathbf{i}_J',\\ &\sum\limits_{\mathbf{i}_J\in\mathfrak{I}_J}m_{\mathbf{i}_J\,\mathbf{i}_J'}=1
\text{\ for all\ }\mathbf{i}_J'\qquad\text{\ and \ }\qquad
\sum\limits_{\mathbf{i}_J'\in\mathfrak{I}_J}m_{\mathbf{i}_J\,\mathbf{i}_J'}=1 \text{\ for all\ }\mathbf{i}_J.
\end{split}
\end{equation*}
The set of doubly stochastic matrices is a convex polytope known as Birkhoff’s polytope \cite{Birkhoff}. Denote by $\mathcal{DS}_n$ this polytope. Set $p=\left[p_{\mathbf{i}_J\,\mathbf{i}_J'}\right]$, where $p_{\mathbf{i}_J\,\mathbf{i}_J^\prime}=\frac{1}{n}$ for all $\mathbf{i}_J,\,\mathbf{i}_J'$. A routine verification demonstrates that $p$ is a {\it minimal orthogonal projection} from $M_{1J}$. If $m=\left[ m_{\mathbf{i}_J\,\mathbf{i}_J^\prime} \right]\in\mathcal{DS}_n$ then
\begin{eqnarray}\label{form_of_double_stochastic}
mp=pm=p \;\text{  and }\; m=p+(I-p)m(I-p).
\end{eqnarray}

A natural method of producing a doubly stochastic matrix is to start with a unitary matrix $u=\left[u_{\mathfrak{i}_J\mathfrak{k}_J}\right]$ and then to set $\mu(u)=\left[\left|u_{\mathfrak{i}_J\mathfrak{k}_J}\right|^2\right]\in
\mathcal{DS}_n$. The matrices of the form $\mu(u)$ with $u$ unitary are called {\it unistochastic}.

It is well known that for $n>3$ there are doubly stochastic matrices that are
not unistochastic \cite{Olkin}.

Let the notation $G$ stand for the set of those $g=\left[g_{\mathbf{i}_J\,\mathbf{i}_J'}\right]\in GL(n,\mathbb{R})$ which satisfy $\sum\limits_{\mathbf{i}_J\in\mathfrak{I}_J}g_{\mathbf{i}_J\,\mathbf{i}_J'}=1$ for all $\mathbf{i}_J'\in\mathfrak{I}_J$ and $\sum\limits_{\mathbf{i}_J'\in
\mathfrak{I}_J}g_{\mathbf{i}_J\,\mathbf{i}_J'}=1$ for all $\mathbf{i}_J\in\mathfrak{I}_J$. The latter relations are obviously equivalent to the vector $\left(\begin{smallmatrix}1\\ 1\\ \vdots\\ 1\end{smallmatrix}\right)$ being invariant under both $g$ and the transpose $g^t$ with respect to matrix multiplication, hence $G$ is a subgroup. One can clearly reproduce (\ref{form_of_double_stochastic}) for $g\in G$:
\begin{eqnarray}
g=p+(I-p)g(I-p).
\end{eqnarray}
Consider the one parameter family  $\,^\theta\!U=\left[ \,^\theta\!U_{\mathbf{i}_J\,\mathbf{i}_J^\prime} \right]$ of unitary matrices, where
\begin{eqnarray}\label{U_theta_formula}
\,^\theta\!U_{\mathbf{i}_J\,\mathbf{i}_J'}=
\delta_{\mathbf{i}_J\,\mathbf{i}_J'}+\frac{\theta-1}{n},\qquad
\theta\in\mathbb{T}=\left\{z\in\mathbb{C}:|z|=1\right\}.
\end{eqnarray}
Now we are in a position to apply the above idea of the present Section \ref{mu_section} in order to introduce the map $\mu:\operatorname{Inn}M\to\mathcal{DS}_n$ given by
$$
\operatorname{Ad}U\mapsto \left[\left|U_{\mathbf{i}_J\,\mathbf{i}_J'}\right|^2\right],\text{\ where\ } U=\left[U_{\mathbf{i}_J\,\mathbf{i}_J'}\right].
$$

An easy calculation demonstrates that
\begin{eqnarray}\label{matrix_U_unistoh}
\mu\left( \,^\theta\!U \right)=p+\left(1-\frac{|\theta-1|^2}{n}\right)(I-p).
\end{eqnarray}
We need below the following claim which is proved in Section \ref{U_to_DS}.

\begin{Prop}\label{openess_mu}
With $\theta\in\mathbb{T}\setminus\{-1,1\}$ and $n>4$, there exists an open neighborhood $\mathcal{U}$ of $\,^\theta\!U $ such that $\mu(\mathcal{U})$ is open in $G$.
\end{Prop}

\subsection{The commutant of $\,^k\!E\cdot\mathfrak{N}^{\otimes k}\left(\operatorname{Ad}\,U(M)\right)\cdot\,^k\!E$.}

Let us start with observing that, in view of \eqref{Cond_expect_tensor_product}, $\,^k\!E\left(L_0^{\otimes k}\right)=\left(  L_0^\mathfrak{A} \right)^{\otimes k}$. It follows that $\,^k\!E\cdot\mathfrak{N}^{\otimes k}\left(\operatorname{Ad}\,U(M)\right) \left(  L_0^\mathfrak{A} \right)^{\otimes k}\subset \left(  L_0^\mathfrak{A} \right)^{\otimes k}$. Thus we can view $\,^k\!E\cdot\mathfrak{N}^{\otimes k}\left(\operatorname{Ad}\,U(M)\right) \cdot \,^k\!E$ as a family of operators on $\left(  L_0^\mathfrak{A} \right)^{\otimes k}$. Finally, let us restrict the representation $\;^k\!\mathcal{P}$ from \ref{action_of+symmetric_group} of $\mathfrak{S}_k$ to the subspace $\left(L_0^\mathfrak{A} \right)^{\otimes k}$, to be denoted by $\;^k\!\mathcal{P}_0^\mathfrak{A}$.

Let $\mathcal{N}_0$ be the $w^*$-algebra generated by the operators $\,^k\!E\cdot\mathfrak{N}^{\otimes k}\left(\operatorname{Ad}\,U(M)\right)\cdot\,^k\!E$ in $\left(L_0^\mathfrak{A}\right)^{\otimes k}$.

\begin{Prop}\label{commutants_mutually}
$\mathcal{N}_0$ coincides with $\left(\;^k\!\mathcal{P}_0^\mathfrak{A}\left(\mathfrak{S}_k\right)\right)'$.
\end{Prop}

We need an auxiliary

\begin{Lm}\label{cond_expect_finite_factor}
Let $\,^k\!\mathfrak{E}_J^p$ $\left( p<J \right)$ be the conditional expectation of $M^{\otimes k}$ onto the ${\rm I}_N$-subfactor $M_{pJ}^{\otimes k}=\left(\left(\left\{ N_l
\right\}_{l=p}^{J}\right)^{\prime\prime}\right)^{\otimes k}$ with respect to ${\rm tr}^{\otimes k}$, where $N=k^{J-p+1}$. Then $\,^k\!\mathfrak{E}_J^p$ belongs to the $w^*$-algebra generated by $\mathfrak{N}^{\otimes k}\left( {\rm Ad}\,u \right)$ with $u$ spanning the unitary group of $w^*$-algebra $\mathfrak{N}\left\{   N_1N_2\cdots N_{p-1}N_{J+1}N_{J+2}\cdots\right\}^{\prime\prime}$.
\end{Lm}

\begin{proof}
Notice first that
\begin{eqnarray}\label{equality_relative_commutant}
M_{pJ}^\prime\cap M=\left\{   N_1N_2\cdots N_{p-1}N_{J+1}N_{J+2}\cdots\right\}^{\prime\prime}.
\end{eqnarray}
Every $x\in M$ can be written in the form $x=\sum\limits_{r,q=1}^N a_{rq}\,x^\prime_{rq}$, where $a_{rq}\in M_{pJ}$, $x^\prime_{rq}\in M_{pJ}^\prime$. Set $\mathfrak{E}_J^p(x)=\sum\limits_{r,q=1}^N{\rm tr}\left(x^\prime_{rq}  \right)\,a_{rq}$. The uniqueness of conditional expectations implies
\begin{eqnarray}\label{Cond_expect_p_J}
\,^k\!\mathfrak{E}_J^p\left( \,^1\!x\otimes\,^2\!x \otimes\cdots\otimes\,^k\!x  \right)= \mathfrak{E}_J^p(\,^1\!x)\otimes\mathfrak{E}_J^p(\,^2\!x) \otimes\cdots\otimes\mathfrak{E}_J^p(\,^k\!x)
\end{eqnarray}
for any $\,^1\!x,\,^2\!x,\ldots,\,^k\!x\in M$. Let $\left\{  j_l \right\}$ and  $\left\{  J_l \right\}$ be two increasing sequences of positive integers with the property
\begin{eqnarray}
J_{l+1}-j_{l+1}>\max\{J_l,J\} \text{\ for all\ }l.
\end{eqnarray}
By (\ref{equality_relative_commutant}), there exists a sequence $\left\{U_l \right\}$ of unitaries from $M_{pJ}^\prime\cap M$ such that
\begin{eqnarray}\label{U_l}
U_{l}\in M_{pJ}^\prime\cap M_{1J_{l+1}}\text{\ and\ } \operatorname{Ad}U_l\left(M_{pJ}'\cap M_{1J_{l}}\right)\subset M_{j_{l+1}\,J_{l+1}}.
\end{eqnarray}
Therefore, $$\text{w-}\!\!\!\lim\limits_{n\to\infty}\operatorname{Ad}U_n(x)=
\operatorname{tr}(x)I~\text{ for each}~ x\in\bigcup\limits_{r=1}^\infty M_{1r}\cap M_{pJ}^\prime,$$ where $\text{w-}\!\!\!\lim\limits_{n\to\infty}x_n$ denote the limit of the sequence $x_n\in M$ in the weak operator topology. Since $\bigcup\limits_{r=1}^\infty M_{1r}$ is dense in $M$ with respect to the strong operator topology, one has
\begin{eqnarray*}
\text{w-}\!\!\!\lim\limits_{n\to\infty}\operatorname{Ad}U_n(x)=
\operatorname{tr}(x)I\text{\ for each\ }x\in M_{pJ}'\cap M.
\end{eqnarray*}
Now, in view of the above observations, with $x=\sum\limits_{r,q=1}^N a_{pq}\,x'_{rq}\in M$, $a_{rq}\in M_{pJ}$,  $x'_{rq}\in M_{pJ}'\cap M$, one establishes that
\begin{eqnarray*}
\text{w-}\!\!\!\lim\limits_{n\to\infty}\operatorname{Ad}U_n(x)=
\sum\limits_{r,q=1}^N\operatorname{tr}\left(x'_{rq}\right)a_{rq}=
\mathfrak{E}_J^p(x)\in M_{pJ}.
\end{eqnarray*}
Hence
$$
\text{w-}\!\!\!\lim\limits_{n\to\infty}\mathfrak{N}^{\otimes k}\left(\operatorname{Ad}U_n\right)\left(\,^1\!x\otimes\,^2\!x \otimes\cdots\otimes\,^k\!x\right)=
\mathfrak{E}_J^p\left(\,^1\!x\right)\otimes
\mathfrak{E}_J^p\left(\,^2\!x\right)\otimes\cdots\otimes
\mathfrak{E}_J^p\left(\,^k\!x\right).
$$
Now combine the latter with (\ref{Cond_expect_p_J}) and (\ref{U_l}) to establish the claim of Lemma \ref{cond_expect_finite_factor}.
\end{proof}

\begin{proof}[\bf Proof of Proposition \ref{commutants_mutually}]
Note first that the conditional expectations $\,^k\!E$ and $\,^k\!\mathfrak{E}_J^p$ commute and
\begin{eqnarray}\label{F_J approx  E}
\lim\limits_{J\to\infty}\,^k\!\mathfrak{E}_J^1\circ\;^k\!E=\;^k\!E.
\end{eqnarray}
To simplify the notation, we substitute below $F_J$ for $\,^k\!\mathfrak{E}_J^1\circ\;^k\!E$. The projection $F_J$ is just the conditional expectation of $M^{\otimes k}$ onto $\mathfrak{A}^{\otimes k}\cap M_{1J}^{\otimes k}$ with respect to ${\rm tr}^{\otimes k}$. Since $\,^k\!E \,\left(L_0^{\otimes k}\right)\subset \left( L_0^\mathfrak{A} \right)^{\otimes k}$ and $\,^k\!\mathfrak{E}_J^1 \,\left(L_0^{\otimes k}\right)=L_0^{\otimes k}\cap M_{1J}^{\otimes k}$, one deduces that
\begin{eqnarray}\label{zero_invariance}
F_J\left( L_0^{\otimes k}\right)\subset M_{1J}^{\otimes k}\cap\left( L_0^\mathfrak{A}\right)^{\otimes k}=\left(M_{1J}\cap L_0^\mathfrak{A}\right)^{\otimes k}.
\end{eqnarray}
By Proposition \ref{limits_of_E_J} and Lemma \ref{cond_expect_finite_factor},
\begin{eqnarray}\label{F_J_in_bicommutant}
F_J\in\left(\mathfrak{N}^{\otimes k}\left({\rm Ad}\, U(M) \right)\right)^{\prime\prime}.
\end{eqnarray}
We are about to use the notation $T_J(u)$ for the operator $F_J\cdot\mathfrak{N}^{\otimes k}(\operatorname{Ad}u)\cdot F_J$.
It follows from (\ref{zero_invariance}) that
\begin{eqnarray}
T_J(u)\left( M_{1J}^{\otimes k}\cap  \left(L_0^\mathfrak{A}\right)^{\otimes k} \right)\subset M_{1J}^{\otimes k}\cap  \left(L_0^\mathfrak{A}\right)^{\otimes k}\;\text{ for each unitary } u\in M_{1J}.
\end{eqnarray}
The above observations imply that the action of $T_J(u)$ on $M_{1J}^{\otimes k}\cap \left(L_0^\mathfrak{A}\right)^{\otimes k}$ is determined by (\ref{Operator_unistochastic}).

Denote by $\mathfrak{L}$ an auxiliary representation of the general linear group $GL(n,\mathbb{R})$, with $n=k^J=|\mathfrak{I}_J|$, which coincides with the natural action of $GL(n,\mathbb{R})$ on the complex $n$-dimensional space $M_{1J}\cap \mathfrak{A}$; more precisely, with $g=\left[g_{{\mathbf{i}_J\,\mathbf{i}_J'}}\right]_{\mathbf{i}_J\,\mathbf{i}_J'
\in\mathfrak{I}_J}\in GL(n,\mathbb{R}) $ one has
\begin{eqnarray}\label{L(g)}
\mathfrak{L}(g)\left( \sum\limits_{\mathfrak{i}_J\in\mathfrak{I}_J}\;a_{\mathbf{i}_J}\cdot
\mathfrak{e}_{\mathbf{i}_J\,\mathbf{i}_J} \right)=\sum\limits_{\mathfrak{i}_J\in\mathfrak{I}_J} \sum\limits_{\mathfrak{i}^\prime_J\in\mathfrak{I}_J}\;g_{{\mathbf{i}_J\,\mathbf{i}_J^\prime}}\;a_{\mathbf{i}^\prime_J}\cdot
\mathfrak{e}_{\mathbf{i}_J\,\mathbf{i}_J}.
\end{eqnarray}
Let us introduce the subgroup $\,^I\!GL(n,\mathbb{R})$ formed by such $g\in GL(n,\mathbb{R})$ that $\mathfrak{L}(g){\rm I}={\rm I}$ and $\mathfrak{L}(g^t){\rm I}={\rm I}$, where the vector $\mathrm{I}=
\sum\limits_{\mathfrak{i}_J\in
\mathfrak{I}_J}\mathfrak{e}_{\mathbf{i}_J\,\mathbf{i}_J}$ is just the unit of the algebra $M_{1J}\cap \mathfrak{A}$, and the superscript $t$ stands for passage to the transpose. Given a unitary $u=\sum\limits_{\mathfrak{i}_J,\mathfrak{i}^\prime_J\in\mathfrak{I}_J}\;
u_{\mathbf{i}_J\,\mathbf{i}_J^\prime}\cdot\mathfrak{e}_{\mathbf{i}_J\,\mathbf{i}_J^\prime}\;\in M_{1J}$, the matrix $\mu(u)=\left[\left| u_{\mathbf{i}_J\,\mathbf{i}_J^\prime} \right|^2\right]$ is doubly stochastic. In the case $\mu(u)$ is also invertible one easily deduces from (\ref{L(g)}) that $\mu(u)\in \,^I\!GL(n,\mathbb{R})$, and in view of (\ref{Operator_unistochastic}) one has
\begin{eqnarray}\label{equality_T_L}
T_J(u)=\mathfrak{L}(\mu(u)).
\end{eqnarray}
 $\,^I\!GL(n,\mathbb{R})$ is the intersection of stationary subgroups of a vector ${\rm I}$  with respect to the left action $g\mapsto\mathfrak{L}(g)$ and to the right action $g\mapsto\mathfrak{L}(g^t)$ on $M_{1J}\cap\mathfrak{A}$. Hence it is isomorphic to $GL(n-1,\mathbb{R})$, and
\begin{eqnarray}\label{GL_invariance}
\mathfrak{L}(g)\left(M_{1J}\cap L_0^\mathfrak{A}\right)= M_{1J}\cap L_0^\mathfrak{A} \;\text{ for all } g\in \,^I\!GL(n,\mathbb{R}).
\end{eqnarray}
By (\ref{equality_T_L}) and (\ref{GL_invariance}), the restrictions   $T_J^0(u)$ and $\mathfrak{L}_0(g)$  of $T_J(u)$  and $\mathfrak{L}(g)$, respectively, to $M_{1J}\cap L_0^\mathfrak{A}$ are well defined. We are about to prove that
\begin{eqnarray}\label{bicommutants_T_L}
\left\{ T_J^0(u), \; u\in M_{1J}\cap U(M)\right\}^{\prime\prime}=\left\{ \mathfrak{L}_0^{\otimes k}\left( \,^I\!GL(n,\mathbb{R}) \right)  \right\}^{\prime\prime}.
\end{eqnarray}
Once the latter relation is established, an application of the well known results of classical Schur-Weyl duality (see, for example, \cite{Fulton-Harris}, Lecture 6) allows one to obtain
\begin{eqnarray*}
\left\{ \mathfrak{L}_0^{\otimes k}\left( \,^I\!GL(n,\mathbb{R}) \right)  \right\}''=\left\{ F_J^0\;\;^k\!\mathcal{P}^\mathfrak{A}\left( \mathfrak{S}_k \right)\,F_J^0  \right\}^\prime,
\end{eqnarray*}
and then to deduce that
\begin{eqnarray}\label{commutant_symmetrical_gr}
\left\{ T_J^0(u), \; u\in M_{1J}\cap U(M)\right\}^{\prime\prime}=\left\{ F_J^0\;\;^k\!\mathcal{P}^\mathfrak{A}\left( \mathfrak{S}_k \right)\,F_J^0  \right\}^\prime,
\end{eqnarray}
where $F_J^0$ is the restriction of $F_J$ to $L_0^{\otimes k}$ (see (\ref{zero_invariance})).

Now we turn to proving (\ref{bicommutants_T_L}).

Since, in view of $\left(\mathfrak{N}^{\otimes k}\left({\rm Ad}\, U(M) \right)\right)^{\prime\prime}\subset\left(\;^k\!\mathcal{P}\left( \mathfrak{S}_k \right)  \right)^\prime$ and (\ref{F_J_in_bicommutant}) one has $F_J$ $\in\left(\mathfrak{N}^{\otimes k}\left({\rm Ad}\, U(M) \right)\right)^{\prime\prime}$, it follows that
\begin{eqnarray}\label{F_J_zero_in__N_0}
\,F_J^0 \in \mathcal{N}_0\subset\left( \;^k\!\mathcal{P}^\mathfrak{A}\left( \mathfrak{S}_k \right) \right)^\prime.
\end{eqnarray}
This implies that for each $J$ the operators $F_J^0\;\;^k\!\mathcal{P}^\mathfrak{A}\left( \mathfrak{S}_k \right)\,F_J^0$ determine a unitary representation of $\mathfrak{S}_k$.

One concludes from Proposition \ref{openess_mu} that there exists an open neighborhood $\mathcal{U}$ $\in U(n)$ of $\,^\theta\! U$ such that $\mu(\mathcal{U})$ is an open subset in $\,^I\!GL(n,\mathbb{R})\cong GL(n-1,\mathbb{R})$. Hence, an application of (\ref{equality_T_L}) yields
$$
T_J^0(\mathcal{U})=\mathfrak{L}_0^{\otimes k}(\mu(\mathcal{U}))\subset\left\{\left.T_J^0(u)\right|\;u\in M_{1J}\cap U(M)\right\}''.
$$
Therefore, with $\mathcal{U}\cdot\mathcal{U}^{-1}$ being a neighborhood of the identity in $U(n)$,
\begin{eqnarray}\label{belonging}
 \mathfrak{L}_0^{\otimes k}\left(\mu(\mathcal{U})\cdot\mu(\mathcal{U})^{-1}\right)
\ \subset\left\{\left.T_J^0(u)\right|\;u\in M_{1J}\cap U(M)\right\}''.
\end{eqnarray}
Denote by $\,^I\!\mathfrak{gl}(n,\mathbb{R})$ and $\mathfrak{gl}(n-1,\mathbb{R})$ the Lie algebras of $\,^I\!GL(n,\mathbb{R})$ and $GL(n-1,\mathbb{R})$, respectively.

A representation $\mathfrak{L}_0^{\otimes k}$ restricted to the neighborhood $\mu(\mathcal{U})\cdot\mu(\mathcal{U})^{-1}$ of unit in $\,^I\!GL(n,\mathbb{R})\cong GL(n-1,\mathbb{R})$ determines a representation $\mathfrak{l}_0^{\otimes k}$ of Lie algebra $\,^I\!\mathfrak{gl}(n,\mathbb{R})\cong \mathfrak{gl}(n-1,\mathbb{R})$ in the $(n-1)^k$-dimensional vector space $M_{1J}^{\otimes k}\cap \left(L_0^\mathfrak{A}\right)^{\otimes k}$. By (\ref{belonging}),
  \begin{eqnarray*}
   \mathfrak{l}_0^{\otimes k}\left(\,^I\!\mathfrak{gl}(n,\mathbb{R}) \right)\subset \left\{ T_J^0(u), \; u\in M_{1J}\cap U(M)\right\}^{\prime\prime}.
  \end{eqnarray*}
This implies (\ref{bicommutants_T_L}).

Consider a bounded operator $B'\in \mathcal{N}_0'$ together with its action on $\left(L_0^\mathfrak{A}\right)^{\otimes k}$. It follows from (\ref{F_J_zero_in__N_0}) that $F^0_JB'=B'F^0_J$. Therefore $B'_J\stackrel{\mathrm{def}}{=}F^0_JB'F^0_J$ belongs to $\left\{\left.T_J^0(u)\right|\;u\in M_{1J}\cap U(M)\right\}'$. Let $R_\lambda$, $\lambda\in\Upsilon_k$, be an irreducible representation of $\mathfrak{S}_k$ and $\chi_\lambda$ its character. Then the operator $P_0^\lambda=\frac{\dim\lambda}{k!}\sum\limits_{s\in\mathfrak{S}_k}
\chi_\lambda(s)\mathcal{P}^\mathfrak{A}_k(s)$ is an orthogonal projection that belongs to the center of $\left(\;^k\!\mathcal{P}^\mathfrak{A}\left(\mathfrak{S}_k\right)\right)'$.

One can readily find such positive integer $N$ that for all $J>N$ one has $F_JP_0^\lambda\ne 0$. Only such $J$ are to be considered below.

It is clear that $P_0^\lambda\in \mathcal{N}_0'$. In view of (\ref{commutant_symmetrical_gr}),
\begin{equation}\label{finite_commutant}
\begin{split}
&B_J'=\sum\limits_{g\in\mathfrak{S}_k}c_J(g)\;F_J^0\;\;^k\!\mathcal{P}^\mathfrak{A}
\left(g\right)\,F_J^0,\text{\ where\ }c_J(g)\in\mathbb{C},\;\text{\ and\ }\;
\\ & P_0^\lambda\,B_J'=B_J'\,P_0^\lambda\text{\ for all sufficiently large\ } \;J.
\end{split}
\end{equation}
It also follows from (\ref{commutant_symmetrical_gr}) that
$$
(F_J^0\mathcal{N}_0F_J^0)'= F_J^0\left\{\;^k\!\mathcal{P}^\mathfrak{A}\left(\mathfrak{S}_k\right)
\,\right\}''F_J^0.
$$
Hence, since $P_0^\lambda$, which is central in $\left(\;^k\!\mathcal{P}^\mathfrak{A}\left(\mathfrak{S}_k\right)\right)'$ and commutes with $F_J^0\in \mathcal{N}_0$, one has
$$
\left(P_0^\lambda\,F_J^0\,\mathcal{N}_0\,\,F_J^0P_0^\lambda\right)'=
F_J^0P_0^\lambda\left\{\;^k\!\mathcal{P}^\mathfrak{A}\left(\mathfrak{S}_k\right)
\,\right\}''P_0^\lambda\,F_J^0.
$$
Therefore, $\left(P_0^\lambda\,F_J^0\,\mathcal{N}_0\,P_0^\lambda\,F_J^0\right)^\prime$ is a finite ${\rm I}_{{\rm dim}\,\lambda}$-factor for all $J$ large enough. This implies that the map
$$
F_{\widehat{J}}^0P_0^\lambda\left\{\;^k\!\mathcal{P}^\mathfrak{A}
\left(\mathfrak{S}_k\right)\,\right\}''P_0^\lambda\,F_{\widehat{J}}^0\ni A\mapsto F_J^0\,A\,F_J^0\in F_J^0P_0^\lambda\left\{\;^k\!\mathcal{P}^\mathfrak{A}
\left(\mathfrak{S}_k\right)\,\right\}''P_0^\lambda\,F_J^0
$$
is an isomorphism for $\widehat{J}>N$. Hence an application of (\ref{finite_commutant}) yields
  \begin{eqnarray*}
  P_0^\lambda\,B_{\widehat{J}}^\prime=P_0^\lambda\sum\limits_{g\in\mathfrak{S}_k}c_J(g)\;F_{\widehat{J}}^0\;\;^k\!\mathcal{P}^\mathfrak{A}\left( g \right)\,F_{\widehat{J}}^0.
  \end{eqnarray*}
  Now, using (\ref{F_J approx  E}), after the passage to the limit $\widehat{J}\to\infty$ we obtain
  \begin{eqnarray*}
  P_0^\lambda\,B^\prime=P_0^\lambda\sum\limits_{g\in\mathfrak{S}_k}c_J(g)\;
  \;^k\!\mathcal{P}^\mathfrak{A}\left( g \right) \text{ for all } \lambda\in \Upsilon_k.
  \end{eqnarray*}
  Therefore, $B^\prime=\sum\limits_{g\in\mathfrak{S}_k}c_J(g)\;
  \;^k\!\mathcal{P}^\mathfrak{A}\left( g \right)\in \left(\;^k\!\mathcal{P}^\mathfrak{A}\left( \mathfrak{S}_k\right)\right)^{\prime\prime}$, which completes the proof of proposition \ref{commutants_mutually}.
\end{proof}

\subsection{The cyclicity of $\mathfrak{N}^{\otimes k}(\operatorname{Inn}\,M)\left(\left( L_0^\mathfrak{A}\right)^{\otimes k} \right)$ in $L_0^{\otimes k}$.}\label{cyclicsty}

Denote by $\mathcal{H}$ the closure of the linear span of $\mathfrak{N}^{\otimes k}({\rm Inn}\,M)\left(\left(  L_0^\mathfrak{A} \right)^{\otimes k} \right)$ in $L_0^{\otimes k}$. Our claim to be proved below is that $\mathcal{H}$ coincides with $L_0^{\otimes k}$.

Let us keep the notation $\left\{N_l\right\}_{l=1}^\infty$ introduced at the beginning of Section \ref{pr_th_2}; let also $\left\{\,^n\!e_{ij}\right\}_{i,j=1}^k\subset N_n$ stand for the collection of matrix units of $N_n$. Denote by $\,^n\!p^s_1$, $s\in\mathfrak{S}_k$, the projection
$$
\;^k\!\mathcal{P}(s)\left( \,^n\!e_{11}\otimes \,^n\!e_{22}\otimes \ldots \otimes \,^n\!e_{kk}\right)\in M^{\otimes k}\subset L^2(M^{\otimes k},{\rm tr}^{\otimes k}).
$$
Set $\,^n\!E_1=\sum\limits_{s\in\mathfrak{S}_k}\,^n\!p^s_1$ and $\,^n\!p^s_2=\left({\rm I}- \,^n\!E_1 \right)\cdot\,^{(n+1)}\!p^s_1$. Proceed with this construction by introducing $\,^n\!p^s_{i+1}=\left({\rm I}- \,^n\!E_i  \right)\cdot\,^{(n+i)}\!p^s_i$ and $\,^n\!E_{i+1}=\,^n\!E_i+ \sum\limits_{s\in\mathfrak{S}_k}\,^n\!p^s_{i+1}$. It is clear that the projections $\,^n\!p^s_m$ are pairwise orthogonal. Introduce
$$
\,^n\!E_m=\sum\limits_{j=1}^m\sum\limits_{s\in\mathfrak{S}_k}\,^n\!p^{s}_j,
$$
and $\tau_i=\mathrm{tr}^{\otimes k}\left(\,^n\!E_i\right)$, which is certainly an increasing sequence. One can readily compute that $\tau_{i+1}=\tau_i+(1-\tau_i)\frac{k!}{k^k}$, whence
$$
\lim\limits_{i\to\infty}\mathrm{tr}^{\otimes k}\left(\,^n\!E_i\right)=1.
$$
This implies
\begin{equation}\label{full_system_orthoprojections}
\sum\limits_{j=1}^\infty\sum\limits_{s\in\mathfrak{S}_k}\,^n\!p^{s}_j=I.
\end{equation}
due to faithfulness of the trace $\mathrm{tr}^{\otimes k}$.

\begin{Lm}\label{cyclic_lemma}
Let $A_1,A_2,\ldots, A_k$ be a family of selfadjoint operators in $M_{1J}$. Set $A=A_1\otimes A_2\otimes\cdots\otimes A_k$. Then for any pair of positive integers $m,n$ with $n>J$, and any $s\in\mathfrak{S}_k$ there exists a unitary $U\in M$ such that $\operatorname{Ad}U\left(A\,^n\!p^{s}_m\right)\in\mathfrak{A}^{\otimes k}$.
\end{Lm}

\begin{proof}
Note that
\begin{equation}\label{Orthopat_operato}
\begin{split}
&A\cdot\,\,^n\!p^{s}_m=\left(\mathrm{I}-\,^n\!E_{m-1}\right)\left(B_1\otimes B_2\otimes\cdots\otimes B_k\right),\text{\  where\ }
\\ &B_i=A_i\cdot\;^{(n+m-1)}\!e_{s^{-1}(i)\,s^{-1}(i)}.
\end{split}
\end{equation}
There exists unitary $U_i\in M_{1J}$ such that
\begin{eqnarray}\label{diagonal_form}
U_i\,A_i\,U_i^*\in\mathfrak{A}\cap M_{1j}.
\end{eqnarray}
Since $n>J$, the operator $\,^n\!U^s_m=\sum\limits_{i=1}^k U_i\cdot\;^{(n+m-1)}\!e_{s^{-1}(i)\,s^{-1}(i)}$ is unitary. By (\ref{Orthopat_operato}) and ({\ref{diagonal_form}}),\ \  $\mathfrak{N}^{\otimes k}({\rm Ad}\,\,^n\!U^s_m)\left(A \cdot\,\,^n\!p^{s}_m\right)\,\in\mathfrak{A}^{\otimes k}$.
\end{proof}

\begin{Co}\label{ciclicity_corollary}
Let $A$ be the same as in Lemma \ref{cyclic_lemma}. Then $A$ belongs to the closed linear span of the collection of operators $\left\{\mathfrak{N}^{\otimes k }(\operatorname{Ad}u)\left(\mathfrak{A}^{\otimes k}\right)\right\}_{u\in U(M)}$ with respect to the norm topology of the space $L^2\left(M^{\otimes k},\mathrm{tr}^{\otimes k}\right)$.
\end{Co}

\begin{proof}
One deduces from (\ref{full_system_orthoprojections}) that
\begin{eqnarray*}
A=\sum\limits_{j=1}^\infty\sum\limits_{s\in\mathfrak{S}_k} A\cdot \,^n\!p^s_{j}.
\end{eqnarray*}
Hence, an application of Lemma \ref{cyclic_lemma} proves our claim.
\end{proof}
\subsection{Proof of Theorem \ref{main_Schur-Weyl}. } Let $\mathfrak{A}$ be a Cartan MASA in $M$ introduced  the beginning of section \ref{pr_th_2}. For convenience, we recall the notations used above:
\begin{eqnarray*}
L_0=\left\{v\in L^2( M,{\rm tr}):{\rm tr}(v)=0  \right\},\;\; L_0^\mathfrak{A}=\left\{x\in L^2\left(\mathfrak{A},\mathrm{tr}\right):\:\mathrm{tr}(x)=0\right\}.
\end{eqnarray*}
We denote  by   $\mathfrak{N}_0^{\otimes k}$ the restriction of $\mathfrak{N}^{\otimes k}$ to $L_0^{\otimes k}$.
Conditional expectation  $\,^k\!E$ introduced in section \ref{CE} is at the same time an orthogonal projection of $L^2( M^{\otimes k},{\rm tr}^{\otimes k})$ onto $L^2\left(\mathfrak{A}^{\otimes k},{\rm tr}^{\otimes k}\right)$ and
\begin{eqnarray}
\,^k\!E\,L_0^{\otimes k}=\left(L_0^\mathfrak{A}\right)^{\otimes k}
\end{eqnarray}
By proposition \ref{commutants_mutually},
\begin{eqnarray}\label{Sch_W}
\left(\,^k\!E\cdot\mathfrak{N}_0^{\otimes k}\left(\operatorname{Ad}\,U(M)\right)\cdot\,^k\!E\right)'=
\left(\;^k\!\mathcal{P}_0^\mathfrak{A}\left(\mathfrak{S}_k\right)\right){''},
\end{eqnarray}
where $\;^k\!\mathcal{P}_0^\mathfrak{A}$ is a restriction of the representation $\;^k\!\mathcal{P}$ (see (\ref{action_of+symmetric_group})) to the subspace $\left(L_0^\mathfrak{A}\right)^{\otimes k}$.

Take any operator $B'\in\left(\mathfrak{N}_0^{\otimes k}\left(\operatorname{Ad}\,U(M)\right)\right)'$.
It follows from Proposition \ref{limits_of_E_J}  that $\,^k\!E\in\left(\mathfrak{N}^{\otimes k}\left(\operatorname{Ad}\,U(M)\right)\right)^{\prime\prime}$. Hence, using (\ref{Sch_W}), we have
\begin{eqnarray}
\,^k\!E\cdot B'\cdot \,^k\!E= B'\cdot \,^k\!E=\,^k\!E\cdot B'\in \left(\;^k\!\mathcal{P}_0^\mathfrak{A}\left(\mathfrak{S}_k\right)\right){''}.
\end{eqnarray}
It follows from Corollary \ref{ciclicity_corollary} that the maps
\begin{eqnarray*}
&\left(\mathfrak{N}_0^{\otimes k}\left(\operatorname{Ad}\,U(M)\right)\right)'\ni X'\stackrel{\Theta}{\mapsto}\,^k\!E\,X'\in \left(\mathfrak{N}_0^{\otimes k}\left(\operatorname{Ad}\,U(M)\right)\right)'\,^k\!E,\\
&\left(\;^k\!\mathcal{P}_0^\mathfrak{A}\left(\mathfrak{S}_k\right)\right){''}\ni X'\stackrel{\Phi}{\mapsto}\,^k\!E\,X'\in \left(\;^k\!\mathcal{P}_0^\mathfrak{A}\left(\mathfrak{S}_k\right)\right){''}
\end{eqnarray*}
are isomorphisms. Hence, using the equality $$\left(\mathfrak{N}_0^{\otimes k}\left(\operatorname{Ad}\,U(M)\right)\right)'\,^k\!E\,\stackrel{(\ref{Sch_W})}{=}
\left(\;^k\!\mathcal{P}_0^\mathfrak{A}\left(\mathfrak{S}_k\right)\right){''},$$ we get that $B'\in \left(\;^k\!\mathcal{P}_0^\mathfrak{A}\left(\mathfrak{S}_k\right)\right){''}$. Theorem \ref{main_Schur-Weyl} is proven.\qed
\section{The Schur-Weyl duality for automor\-phisms group of factor and the symmetric inverse se\-migroup}\label{finite_inverse_semigroup}
{\it The symmetric inverse semigroup} $\mathscr{I}_k$ is formed by all the partial {\it bijections} from the set $X_k=\left\{1,2,\ldots,k\right\}$ to itself, with the natural definition of multiplication. An element $\mathbf{b}\in\mathscr{I}_m$ is conveniently written as $\mathbf{b}=\left(\begin{matrix}i_1&i_2&\ldots &i_r\\j_1&j_2&\ldots &j_r\end{matrix}\right)$, where $\left\{i_1,i_2,\ldots,i_r\right\}\subset X_k$, $\left\{j_1,j_2,\ldots,j_r\right\}\subset X_k$ and $i_l$ maps to $j_l$. The number $r$ involved here is denoted by $\operatorname{rank}\mathbf{b}$. There exists a natural involution on $\mathscr{I}_k$:  $\mathbf{b}^*=\left(\begin{matrix}j_1&j_2&\ldots &j_r\\i_1&i_2&\ldots &i_r\end{matrix}\right)$. Denote by $\mathrm{id}_\mathcal{A}\in \mathscr{I}_m$ the partial bijection obtained by restricting the identity map to $\mathcal{A}\subset X_k$; introduce also the abbreviation $\epsilon_j=\mathrm{id}_{\left(X_m\setminus\{j\}\right)}$. The subcollection $\left\{\mathbf{b}\in\mathscr{I}_k:\operatorname{rank}\mathbf{b}=k\right\}$ is just the ordinary symmetric group $\mathfrak{S}_k$.

Let $\left\{s_i\right\}_{i=1}^{k-1}$ be the collection of Coxeter generators of $\mathfrak{S}_k$, where $s_i=(i\;i+1)$ is the transposition of $i$ and $i+1$. The following claim is due to L. Popova \cite{Popova}. A more up-to-date exposition of her results is given in \cite{East}.

\begin{Th}[{\bf A description of $\mathscr{I}_m$ in the terms of the generators and the relations}]\label{Popova_Th}\hfill\\
The semigroup $\mathscr{I}_k$ is generated by $\left\{s_i\right\}_{i=1}^{k-1}$ and $\epsilon_1$ with the relations as follows:
\begin{description}
\item[a)] the Coxeter relations for $\left\{s_i\right\}_{i=1}^{k-1}$;
\item[b]) $s_i\,\epsilon_1=\epsilon_1\,s_i$ for all $i>1$;
\item[c)] $(s_1\,\epsilon_1)^2=(\epsilon_1\,s_1)^2=\epsilon_1\,s_1\,\epsilon_1$.
\end{description}
\end{Th}

This implies that one can realize $\mathscr{I}_k$ as a semigroup of $\{0,1\}$-matrices $a=\left[  a_{ij} \right]$ with the ordinary matrix multiplication in such a way that $a$ has at most one nonzero entry in each row and each column. The matrix $a=\left[  a_{ij} \right]$, where $a_{11}=0$ and $a_{ij}=\delta_{ij}$, if $i\neq 1$ or $j\neq 1$, corresponds to $\epsilon_1$ under this realization.

Let $\mathbb{C}\left[ \mathfrak{S}_k \right]$ be the complex group algebra of the symmetric group $\mathfrak{S}_k$. This algebra  as well as  the group algebra of every finite group, is semisimple.
The complex semigroup algebra $\mathbb{C}\left[ \mathscr{I}_k \right]$ of the inverse symmetric semigroup is semisimple too. Namely, Munn proved the next statement.
\begin{Th}[\cite{Munn_2}]\label{Munn_th}
The algebra $\mathbb{C}\left[ R_k \right]$ has the decomposition
$$\mathbb{C}\left[ R_k \right]=\bigoplus\limits_{l=0}^k \mathbb{M}_{\binom{k}{l}}\left( \mathbb{C}\left[ \mathfrak{S}_l \right] \right),$$
where $\mathbb{M}_j(A)$ is the algebra of all $j\times j$-matrices over an algebra $A$.
\end{Th}

Denote by $\Upsilon_m$ the set of all unordered partitions of  a positive integer $m\leq k$. It follows from  previous theorem that the set of the irreducible representations of the semigroup $R_k$ can be naturally indexed by the set $\bigcup\limits_{m=0}^k\Upsilon_m$.
\subsection{The action of $\mathscr{I}_k$ on $L^2\left(M^{\otimes k},\mathrm{tr}^{\otimes k}\right)$.}\label{section_action}

Consider the operators $\;^k\!\mathcal{P}^\mathscr{I}(\epsilon_i)$ on $L^2\left(M^{\otimes k},\mathrm{tr}^{\otimes k}\right)$:
\begin{multline}
\;^k\!\mathcal{P}^\mathscr{I}(\epsilon_i)\left(\cdots\otimes v_{i-1}\otimes v_i\otimes v_{i+1}\otimes\cdots\right)
\\ =\mathrm{tr}\left(v_i\right)\left(\cdots\otimes v_{i-1}\otimes\mathrm{I}\otimes v_{i+1}\otimes\cdots\right).
\end{multline}
Set also $\;^k\!\mathcal{P}^\mathscr{I}(s)=\;^k\!\mathcal{P}(s)$ with $s\in\mathfrak{S}_k$, see (\ref{action_of+symmetric_group}). Theorem \ref{Popova_Th} implies that $\;^k\!\mathcal{P}^\mathscr{I}$ admits an extension to a representation of $\mathscr{I}_k$. One has the following obvious  result:
\begin{Prop}\label{compare_bicommutant_inverse}
$\left( \mathfrak{N}^{\otimes k}\left( {\rm Aut}\,M \right) \right)^{\prime\prime}\subset\left(\;^k\!\mathcal{P}^\mathscr{I}(\mathscr{I}_k)\right)^\prime$.
\end{Prop}
Below we prove the next statement, which is the analogue of Schur-Weyl duality for  ${\rm Aut}\,M$ and $\mathscr{I}_k$.
\begin{Th}\label{Inverse_Schur-Weyl}
$\left( \mathfrak{N}^{\otimes k}\left( {\rm Aut}\,M \right) \right)^{\prime\prime}=$ $\left(\;^k\!\mathcal{P}^\mathscr{I}(\mathscr{I}_k)\right)^\prime$.
\end{Th}

\begin{Rem}
The operator $\;^k\!\mathcal{P}^\mathscr{I}(\epsilon_i)$ is an orthogonal projection in $L^2\left( M, {\rm tr} \right)^{\otimes k}$ and
\begin{multline*}
\prod\limits_{i=1}^k\left( {\rm I}-\;^k\!\mathcal{P}^\mathscr{I}(\epsilon_i)\right)L^2\left( M^{\otimes k}, {\rm tr}^{\otimes k} \right)
\\ =\left\{v\in L^2\left( M^{\otimes k}, {\rm tr}^{\otimes k} \right):\;^k\!\mathcal{P}^\mathscr{I}(\epsilon_i)v=0\;\text{ for all }\; i=1,2,\ldots,k\right\}=L_0^{\otimes k}.
\end{multline*}
\end{Rem}
Let $\wp_m(X_k)$ be the collection\footnote{$\wp_0(X_k)$ is the unique empty subset.} of all non-ordered $m$-element subsets of $X_k$. With $\mathcal{A}\in\wp_m(X_k)$, let us introduce the pairwise orthogonal projections $\,^k\!P_\mathcal{A}$  as follows
\begin{eqnarray*}
\,^k\!P_\mathcal{A}=\prod\limits_{j\in X_k\setminus\mathcal{A}}\;^k\!\mathcal{P}^\mathscr{I}(\epsilon_j)\cdot\prod\limits_{j\in \mathcal{A}}\left({\rm I}-\;^k\!\mathcal{P}^\mathscr{I}(\epsilon_j)\right).
\end{eqnarray*}
Hence
\begin{equation}\label{properties}
\begin{aligned}
&\;^k\!\mathcal{P}^\mathscr{I}(\epsilon_j)\,^k\!P_\mathcal{A}=0 &&\text{\ for all\ } j\in \mathcal{A},&
\\ &\;^k\!\mathcal{P}^\mathscr{I}(\epsilon_j)\,^k\!P_\mathcal{A}=\,^k\!P_\mathcal{A}
&&\text{\ for all\ } j\in X_k\setminus\mathcal{A}.&
\end{aligned}
\end{equation}
Since the projections $\,^k\!P_\mathcal{A}$ and $\,^k\!P_\mathcal{B}$ are orthogonal for different $\mathcal{A}$ and $\mathcal{B}$, then operator $\,^k\!P_m=\sum\limits_{\mathcal{A}\in\wp_m(X_k)}\,^k\!P_\mathcal{A}$ is an orthogonal projection.
It is clear that $\,^k\!P_k\,L^2\left(  M^{\otimes k}, {\rm tr}^{\otimes k} \right)=L_0^{\otimes k}$, $\,^k\!P_k\,L^2\left(  M^{\otimes k}, {\rm tr}^{\otimes k} \right)=\mathbb{C}{\rm I}^{\otimes k}$ and $$\sum\limits_{m=0}^k \,^k\!P_m\,L^2\left(  M^{\otimes k}, {\rm tr}^{\otimes k} \right)=L^2\left(  M^{\otimes k}, {\rm tr}^{\otimes k} \right).$$

Let $m\leq k$ and let $\mathfrak{S}_m=\left\{ s\in\mathfrak{S}_k: s(j)=j \text{ for all } j\in X_k\setminus X_m  \right\}$, where $X_m=\{1,2,\ldots, m\}\subset X_k$. Denote by $\chi_\gamma$ the character of the irreducible representation $T_\gamma$ of $\mathfrak{S}_m$, corresponding to $\gamma$ $\in\Upsilon_m$, such that its value on the unit is equal to the dimension of $T_\gamma$. Then Young projection $$P^\gamma=\frac{{\rm dim} \,\gamma }{m!}\sum\limits_{s\in\mathfrak{S}_m}\chi_\gamma(s)\;\;^k\!\mathcal{P}^{\mathscr{I}}(s)$$ lies in the center of $*$-algebra generated by $\;^k\!\mathcal{P}^{\mathscr{I}}\left(\mathfrak{S}_m\right)$. Since $\,^k\!P_{X_m}$ belongs to $\;^k\!\mathcal{P}^{\mathscr{I}}\left(\mathfrak{S}_m\right)^\prime$, then $\,^k\!P_{X_m}^\gamma=\,^k\!P_{X_m}\cdot P^\gamma$ is an orthogonal projection from $\;^k\!\mathcal{P}^{\mathscr{I}}\left(\mathfrak{S}_m\right)^\prime$.
Denote by $\,^k\!\mathcal{H}_m^\gamma$ the closure of the linear span of the set $$\left\{ \;^k\!\mathcal{P}^{\mathscr{I}}\left(\mathscr{I}_k\right)\,^k\!P_{X_m}^\gamma\,L^2\left(  M^{\otimes k}, {\rm tr}^{\otimes k}  \right)  \right\}$$ with respect to the norm topology of the space $L^2\left(M^{\otimes k},\mathrm{tr}^{\otimes k}\right)$. By proposition \ref{compare_bicommutant_inverse}, the $\;^k\!\mathcal{P}^{\mathscr{I}}$-invariant subspace $\,^k\!\mathcal{H}_m^\gamma$ is $\mathfrak{N}^{\otimes k}\left({\rm Aut}\,M \right)$-invariant too.

\subsection{Decomposing $\mathfrak{N}^{\otimes k}$ into factor-components.}
Set $\,^k\!\mathcal{H}_{X_m}=\,^k\!P_{X_m}\,L^2\left( M^{\otimes k},{\rm tr}^{\otimes k} \right)$. By proposition \ref{compare_bicommutant_inverse}, $\,^k\!\mathcal{H}_{X_m}$ is $\mathfrak{N}^{\otimes k}$-invariant. Let $\mathfrak{N}^{\otimes k}_{X_m}$ be the restriction  of $\mathfrak{N}^{\otimes k}$ to $\,^k\!\mathcal{H}_{X_m}$.  Here $m\leq k$ and we consider $X_m=\{1,2\ldots,m \}$ as a subset of $X_k$. Clearly, $\,^k\!\mathcal{H}_{X_m}$ is invariant under the operators $\;^k\!\mathcal{P}(s)$, where $s\in\mathfrak{S}_m\subset\mathfrak{S}_k$, and, more generally,
\begin{eqnarray}\label{imprimitivity_system}
\;^k\!\mathcal{P}(s)\cdot\,^k\!P_{\mathcal{A}}\cdot\;^k\!\mathcal{P}(s^{-1})=
\,^k\!P_{s(\mathcal{A})} \text{ for all } s\in\mathfrak{S}_k \text{ and } \mathcal{A}\in \wp_m(X_k).
\end{eqnarray}
Consider Young subgroup $\mathfrak{S}_{m\,(k-m)}=\left\{s\in\mathfrak{S}_k:sX_m=X_m \right\}$.
Let $s_1$, $s_2$, $\ldots$, $s_r$ be a full set of the representatives in $\mathfrak{S}_k$ of the left cosets $\mathfrak{S}_k\diagup\mathfrak{S}_{m\,(k-m)}$, where $r=|\mathfrak{S}_k\diagup\mathfrak{S}_{m\,(k-m)}|$. Then the projections $\,^k\!P_{s_j(X_m)}$ are pairwise orthogonal and
\begin{eqnarray}\label{K_P_M_decomposition}
\,^k\!P_m=\sum\limits_{j=1}^r\,^k\!P_{s_j(X_m)}.
\end{eqnarray}
 By (\ref{properties}),
\begin{eqnarray}\label{m_homogeneous}
 \mathfrak{N}^{\otimes k}(\theta)\;\;^k\!\mathcal{P}^\mathscr{I}(s)\;\,^k\!P_m= \,^k\!P_m\;\mathfrak{N}^{\otimes k}(\theta)\;\;^k\!\mathcal{P}^\mathscr{I}(s)
\end{eqnarray}
for all  $\theta\in{\rm Aut}\,M$  and  $s\in\mathscr{I}_k$.
 We emphasize again that $\,^k\!P_{X_m}\,\;^k\!\mathcal{P}^\mathscr{I} (\epsilon_j)=0$  for all $j\in X_m$. Therefore,
\begin{eqnarray}
\left( \,^k\!P_{X_m}\,\;^k\!\mathcal{P}^\mathscr{I} (\mathscr{I}_m) \right)^{\prime\prime}=\left( \,^k\!P_{X_m}\,\;^k\!\mathcal{P}^\mathscr{I}(\mathfrak{S}_m) \right)^{\prime\prime}.
\end{eqnarray}
Let $\gamma\in\Upsilon_m$ be an unordered partition of $m$  and let $\chi_\gamma$ be the character of the corresponding irreducible representation of $\mathfrak{S}_m$. Set
\begin{eqnarray}
P^\gamma=\frac{{\rm dim} \,\gamma }{m!}\sum\limits_{s\in\mathfrak{S}_m}\chi_\gamma(s)\;\;^k\!\mathcal{P}^{\mathscr{I}}(s).
\end{eqnarray}
Since the projections $\left\{\,^k\!P_{s_j(X_m)} \right\}_{j=1}^{r}$
are pairwise orthogonal and
\begin{eqnarray*}
\,^k\!P_{X_m}\in \left( \;^k\!\mathcal{P}^\mathscr{I}(\mathfrak{S}_m) \right)^{\prime}~\text{then}~
\,^k\!P_{X_m}^\gamma=P^\gamma\cdot\,^k\!P_{X_m}
\end{eqnarray*} is an orthogonal projection from the center of $w^*$-algebra, generated by the operators $\,^k\!P_{X_m}\mathfrak{N}^{\otimes k}({\rm Aut}\,M)$ and $\,^k\!P_{X_m}\cdot \;^k\!\mathcal{P}^\mathscr{I}(\mathfrak{S}_m) $. Therefore,  the operator
\begin{eqnarray}\label{def_of_P_k_g_m}
\,^k\!P_{m}^\gamma=\sum\limits_{j=1}^{r} \;^k\!\mathcal{P}(s_j)\cdot\,^k\!P_{X_m}^\gamma\cdot\;^k\!\mathcal{P}(s_j^{-1})
\end{eqnarray}  is an orthogonal projection too.  Moreover, the projections $\,^k\!P_{m}^\gamma$  and $\,^k\!P_{m}^{\widetilde{\gamma}}$ are orthogonal for different $\gamma, \widetilde{\gamma}\in\Upsilon_m$ and  the following  equality holds
\begin{eqnarray}\label{m_full_sys_projections}
\,^k\!P_{m}=\sum\limits_{\gamma\in\Upsilon_m}\,^k\!P_{m}^\gamma.
\end{eqnarray}

The next statement follows from theorem \ref{main_Schur-Weyl}.

\begin{Lm}\label{Lemma about restriction}
The family  of the operators  $\left\{\,^k\!P_{X_m} \;\;^k\!\mathcal{P}^{\mathscr{I}}(s)\;  \,^k\!P_{X_m} \right\}_{s\in\mathfrak{S}_m}$  define the unitary representation  $\;^k\!\mathcal{P}^{\mathscr{I}}_{X_m}$ of the group  $\mathfrak{S}_m$  in the subspace $\,^k\!\mathcal{H}_{X_m}$ and one has
$\left( \mathfrak{N}^{\otimes k}_{X_m}({\rm Aut}\,M) \right)^{\prime\prime}=\left( \,^k\!\mathcal{P}^{\mathscr{I}}_{X_m}(\mathfrak{S}_m) \right)^\prime$.
\end{Lm}
Define the representation $\,^k\!\Pi$ of the semigroup $({\rm Aut}\,M)\times \mathscr{I}_k$ as follows
\begin{eqnarray}
\,^k\!\Pi(\theta,s)=\mathfrak{N}^{\otimes k}(\theta)\cdot\;^k\!\mathcal{P}^{\mathscr{I}}(s), \text{ where } \theta\in{\rm Aut}\,M, s\in\mathscr{I}_k.
\end{eqnarray}
\begin{Lm}
Projection $\,^k\!P_{m}^\gamma$ belongs to $w^*$-algebra $\left( \,^k\!\Pi\left( ({\rm Aut}\,M)\times \mathscr{I}_k \right) \right)^\prime$ and the restriction of $\,^k\!\Pi$ to the subspace $\,^k\!P_{m}^\gamma\,L^2\left( M^{\otimes k},{\rm tr}^{\otimes k} \right)$ is the irreducible representation of the semigroup $({\rm Aut}\,M)\times \mathscr{I}_k$.
\end{Lm}
\begin{proof}
Let us prove that
\begin{eqnarray}\label{first_statement}
\,^k\!P_{m}^\gamma\in \left( \,^k\!\Pi\left( ({\rm Aut}\,M)\times \mathscr{I}_k \right) \right)^\prime    \;\;\;\;\;\;\;\;\;\;\;\;\; (\text{see (\ref{def_of_P_k_g_m})}).
\end{eqnarray}
\emph{}Each $t\in\mathfrak{S}_k$ defines the bijection $\mathfrak{b}_t$ of the set $\left\{s_1, s_2, \ldots,s_r  \right\}$, where $r=|\mathfrak{S}_k\diagup\mathfrak{S}_{m\,(k-m)}|$, as follows
\begin{eqnarray*}
\mathfrak{b}_t(s_j)=s_{j_t}, \text{ where } ts_j\in s_{j_t}\mathfrak{S}_{m\,(k-m)}.
\end{eqnarray*}
Hence, since $\,^k\!P_{m}^\gamma=\sum\limits_{j=1}^{|\mathfrak{S}_k\diagup\mathfrak{S}_{m\,(k-m)}|} \;^k\!\mathcal{P}(s_j)\cdot\,^k\!P_{X_m}^\gamma\cdot\;^k\!\mathcal{P}(s_j^{-1})$, then
\begin{align*}
\;^k\!\mathcal{P}(t)\cdot\,^k\!P_{m}^\gamma\cdot \;^k\!\mathcal{P}(t^{-1})=\sum\limits_{j=1}^{|\mathfrak{S}_k\diagup\mathfrak{S}_{m\,(k-m)}|} \;^k\!\mathcal{P}(ts_j)\cdot\,^k\!P_{X_m}^\gamma\cdot\;^k\!\mathcal{P}(s_j^{-1}t^{-1})\\
\sum\limits_{j=1}^{|\mathfrak{S}_k\diagup\mathfrak{S}_{m\,(k-m)}|} \;^k\!\mathcal{P}(\mathfrak{b}_t(s_j)\;h_j)\cdot\,^k\!P_{X_m}^\gamma\cdot\;^k\!\mathcal{P}\left(h_j^{-1}\; (\mathfrak{b}_t(s_j))^{-1}\right), \text{ where } h_j\in \mathfrak{S}_m.
\end{align*}
Now, using the equality $\;^k\!\mathcal{P}(h_j)\cdot\,^k\!P_{X_m}^\gamma\cdot\;^k\!\mathcal{P}\left(h_j^{-1} \right)=\,^k\!P_{X_m}^\gamma$, we obtain
\begin{eqnarray*}
\;^k\!\mathcal{P}(t)\cdot\,^k\!P_{m}^\gamma\cdot \;^k\!\mathcal{P}(t^{-1})=\sum\limits_{j=1}^{|\mathfrak{S}_k\diagup\mathfrak{S}_m|} \;^k\!\mathcal{P}(\mathfrak{b}_t(s_j))\cdot\,^k\!P_{X_m}^\gamma\cdot\;^k\!\mathcal{P}\left( (\mathfrak{b}_t(s_j))^{-1}\right).
\end{eqnarray*}
Since $\mathfrak{b}_t$ is the bijection, then
\begin{align*}
\sum\limits_{j=1}^{|\mathfrak{S}_k\diagup\mathfrak{S}_m|} \;^k\!\mathcal{P}(\mathfrak{b}_t(s_j))\cdot\,^k\!P_{X_m}^\gamma\cdot\;^k\!\mathcal{P}\left( (\mathfrak{b}_t(s_j))^{-1}\right)\\
=\sum\limits_{j=1}^{|\mathfrak{S}_k\diagup\mathfrak{S}_{m\,(k-m)}|} \;^k\!\mathcal{P}(s_j)\cdot\,^k\!P_{X_m}^\gamma\cdot\;^k\!\mathcal{P}\left( s_j^{-1}\right).
\end{align*}
Thus
\begin{eqnarray}\label{Commutant of S_k}
\;^k\!\mathcal{P}(t)\cdot\,^k\!P_{m}^\gamma\cdot \;^k\!\mathcal{P}(t^{-1})=\,^k\!P_{m}^\gamma \text{ for all } t\in\mathfrak{S}_k.
\end{eqnarray}
Set $\mathcal{A}_i$ $=\left\{j\in \{1,2,\ldots, |\mathfrak{S}_k\diagup\mathfrak{S}_{m\,(k-m)}|\} :s_j^{-1}(i)\notin X_m \right\}$. Since \newline $\,^k\!P_{X_m}^\gamma=P^\gamma\cdot\,^k\!P_{X_m} $ $=\,^k\!P_{X_m}\cdot P^\gamma$, then, using (\ref{properties}) and (\ref{imprimitivity_system}), we have $$\;^k\!\mathcal{P}^\mathscr{I}(\epsilon_i)\cdot\,^k\!P_{m}^\gamma=\,^k\!P_{m}^\gamma\cdot
\;^k\!\mathcal{P}^\mathscr{I}(\epsilon_i)= \sum\limits_{j\in\mathcal{A}_i}\;^k\!\mathcal{P}(s_j)\cdot\,^k\!P_{X_m}^\gamma\cdot\;^k\!\mathcal{P}\left( s_j^{-1}\right).$$
Now we conclude from (\ref{Commutant of S_k}) that $\,^k\!P_{m}^\gamma\in\;^k\!\mathcal{P}^\mathscr{I}\left(  \mathscr{I}_k \right)^\prime$. Hence, applying Proposition \ref{compare_bicommutant_inverse}, we obtain (\ref{first_statement}).

Therefore, the operators $
\,^k\!\Pi_m^\gamma(\theta,s)=\,^k\!P_{m}^\gamma\cdot\,^k\!\Pi(\theta,s)$, where $\theta\in{\rm Aut}\,M$, $s\in\mathscr{I}_k$, define $*$-representation of semigroup ${\rm Aut}\,M\times\mathscr{I}_k$.

Let us prove that $\,^k\!\Pi_m^\gamma$ is an irreducible representation; i. e.
\begin{align*}
\,^k\!\Pi_m^\gamma\left({\rm Aut}\,M\times\mathscr{I}_k\right)^\prime=\mathbb{C}\cdot\,^k\!P_{m}^\gamma.
\end{align*}
First, we notice that $\,^k\!P_{X_m}^\gamma\in \,^k\!P_{m}^\gamma\cdot\;^k\!\mathcal{P}^\mathscr{I}\left(  \mathscr{I}_k \right)^{\prime\prime}\subset \,^k\!\Pi_m^\gamma\left({\rm Aut}\,M\times\mathscr{I}_k\right)^{\prime\prime}$. Therefore, if $B^\prime\in\,^k\!\Pi_m^\gamma\left({\rm Aut}\,M\times\mathscr{I}_k\right)^{\prime}$ then
\begin{eqnarray*}
B^\prime\cdot \,^k\!P_{X_m}^\gamma\in \,^k\!P_{X_m}^\gamma\cdot\,^k\!\Pi_m^\gamma\left({\rm Aut}\,M\times\mathscr{I}_k\right)^{\prime}\cdot\,^k\!P_{X_m}^\gamma.
\end{eqnarray*}
Hence, applying Lemma \ref{Lemma about restriction}, we see that
\begin{eqnarray*}
B^\prime\cdot \,^k\!P_{X_m}^\gamma=c\cdot\,^k\!P_{X_m}^\gamma, \text{ where } c\in \mathbb{C}.
\end{eqnarray*}
Now, using (\ref{def_of_P_k_g_m}), we obtain $B^\prime=B^\prime\cdot \,^k\!P_{m}^\gamma=c\cdot\,^k\!P_{m}^\gamma$.
\end{proof}

\subsection{The proof of Theorem \ref{Inverse_Schur-Weyl}.}
Let $B^\prime$ lies in
$\left( \mathfrak{N}^{\otimes k}\left( {\rm Aut}\,M \right) \right)^{\prime}$. For the matrix $\,^\theta\!U=\left[ \,^\theta\!U_{\mathbf{i}_J\,\mathbf{i}_J^\prime} \right]$ (see (\ref{U_theta_formula})), we denote by $\,^\theta\!\mathbf{U}$ an element from $M_{1J}$ of the view
\begin{eqnarray*}
\,^\theta\!\mathbf{U}=\sum\limits_{\mathfrak{i}_J,\mathfrak{i}^\prime_J\in\mathfrak{I}_J}\;
\,^\theta\!U_{\mathbf{i}_J\,\mathbf{i}_J^\prime}\cdot\mathfrak{e}_{\mathbf{i}_J\,\mathbf{i}_J^\prime}.
\end{eqnarray*}
Let $a\in M_{1J}\cap\mathfrak{A}$.
Using (\ref{Operator_unistochastic})  and (\ref{matrix_U_unistoh}), we obtain
\begin{eqnarray*}
\,^k\!E\circ\mathfrak{N}^{\otimes k}\!({\rm Ad}\,\,^\theta\!\mathbf{U})(\,^k\!P_m(a))=\left(1-\frac{|\theta-1|^2}{n}\right)^m\,^k\!P_m(a).
\end{eqnarray*}
It follows that
\begin{align*}
\,^k\!E\circ\mathfrak{N}^{\otimes k}\!({\rm Ad}\,\,^\theta\!\mathbf{U})\circ\,^k\!E\\
=\sum\limits_{j=0}^k\left(1-\frac{|\theta-1|^2}{n}\right)^j\,^k\!E\circ\,^k\!P_j
\in  \left( \mathfrak{N}^{\otimes k}\left( {\rm Aut}\,M \right) \right)^{\prime\prime}.
\end{align*}
Therefore,
\begin{eqnarray*}
\sum\limits_{j=0}^k\left(1-\frac{|\theta-1|^2}{n}\right)^jB^\prime\circ\,^k\!E\circ\,^k\!P_j=
\sum\limits_{j=0}^k\left(1-\frac{|\theta-1|^2}{n}\right)^j\,^k\!E\circ\,^k\!P_j\circ B^\prime
\end{eqnarray*}
Hence, thanks to the relation $\,^k\!P_l\circ\,^k\!P_m=\delta_{ml}\,^k\!P_l$, we have
\begin{align*}
\left(1-\frac{|\theta-1|^2}{n}\right)^m\;\,^k\!P_l\circ B^\prime\circ\,^k\!E\circ\,^k\!P_m\\
= \sum\limits_{j=0}^k\left(1-\frac{|\theta-1|^2}{n}\right)^j\;\,^k\!P_l\circ\;^k\!E\circ\,^k\!P_j\circ B^\prime\circ\,^k\!P_m.
\end{align*}
Now we conclude from propositions \ref{limits_of_E_J}  and \ref{compare_bicommutant_inverse} that
\begin{equation*}
\left(1-\frac{|\theta-1|^2}{n}\right)^m\;\,^k\!P_l\circ B^\prime\circ\,^k\!E\circ\,^k\!P_m=\left(1-\frac{|\theta-1|^2}{n}\right)^l\;\,^k\!P_l\circ \,^k\!E\circ B^\prime\circ\,^k\!P_m
\end{equation*} and
\begin{equation*}
\left(1-\frac{|\theta-1|^2}{n}\right)^m\;\,^k\!P_l\circ B^\prime\circ\,^k\!E\circ\,^k\!P_m=\left(1-\frac{|\theta-1|^2}{n}\right)^l\;\,^k\!P_l\circ B^\prime\circ \,^k\!E\circ \,^k\!P_m.
\end{equation*}
Therefore, $\,^k\!P_l\circ B^\prime\circ\,^k\!E\circ\,^k\!P_m=\delta_{lm}\; ^k\!P_m\circ B^\prime\circ\,^k\!E\circ\,^k\!P_m$. Now, using the relation $\sum\limits_{j=0}^k\,^k\!P_j={\rm I}$, we have
\begin{eqnarray*}
B^\prime\circ\,^k\!E=\,^k\!E\circ B^\prime=\sum\limits_{m=0}^k\; ^k\!P_m\circ B^\prime\circ\,^k\!E\circ\,^k\!P_m.
\end{eqnarray*}
Hence, applying corollary \ref{ciclicity_corollary}, we conclude
\begin{eqnarray}
B^\prime=\sum\limits_{m=0}^k\; ^k\!P_m\circ B^\prime\circ\,^k\!P_m.
\end{eqnarray}
 Let us prove that $B^\prime_m\stackrel{\rm def}{=}\; ^k\!P_m\circ B^\prime\circ\,^k\!P_m$ lies in $*$-algebra $^k\!P_m\;^k\!\mathcal{P}^\mathscr{I}(\mathscr{I}_k)^{\prime\prime}\;^k\!P_m$ (see (\ref{m_full_sys_projections}) and lemma \ref{Lemma about restriction}).

 Since $\;^k\!P_m=\sum\limits_{\mathcal{A}\in\wp_m(X_k)} \,^k\!P_\mathcal{A}$, then $B^\prime_m=\sum\limits_{\mathcal{A},\mathcal{B}\in\wp_m(X_k)}\,^k\!P_\mathcal{A}\circ B^\prime_m \circ \,^k\!P_\mathcal{B}$. There exist $s_\mathcal{A}$, $s_\mathcal{B}\in \mathfrak{S}_k$ such that
 \begin{eqnarray}
 s_\mathcal{A}(X_m)=\mathcal{A}\text{ and } s_\mathcal{B}(X_m)=\mathcal{B}.
 \end{eqnarray}
 Hence, using (\ref{imprimitivity_system}), we have
 \begin{align*}
 ^k\!P_\mathcal{A}\circ B^\prime_m \circ \,^k\!P_\mathcal{B}=\,^k\!\mathcal{P}(s_\mathcal{A})\circ\,^k\!P_{X_m}\circ\,^k\!\mathcal{P}(s_\mathcal{A}^{-1})
 \circ B^\prime_m \circ \,^k\!\mathcal{P}(s_\mathcal{B})\circ\,^k\!P_{X_m}\circ\,^k\!\mathcal{P}(s_\mathcal{B}^{-1}).
 \end{align*}
 It follows from lemma \ref{Lemma about restriction} that $\,^k\!P_{X_m}\circ\;^k\!\mathcal{P}(s_\mathcal{A}^{-1})
 \circ B^\prime_m \circ \;^k\!\mathcal{P}(s_\mathcal{B})\circ\,^k\!P_{X_m}$ lies in algebra $\,^k\!P_{X_m}\circ\;^k\!\mathcal{P}(\mathfrak{S}_m)^{\prime\prime}\circ \,^k\!P_{X_m}$. Therefore,
 \begin{eqnarray*}
 \,^k\!P_\mathcal{A}\circ B^\prime_m \circ \,^k\!P_\mathcal{B}\in \left(\,^k\!\mathcal{P}^{\mathscr{I}}(\mathscr{I}_k)\right)^{\prime\prime}.
 \end{eqnarray*}
 Thus $B^\prime=\sum\limits_{m=0}^k\sum\limits_{\mathcal{A},\mathcal{B}\in\wp_m(X_k)}\,^k\!P_\mathcal{A}\circ B^\prime_m \circ \,^k\!P_\mathcal{B}$ lies in $\left(\,^k\!\mathcal{P}^{\mathscr{I}}(\mathscr{I}_k)\right)^{\prime\prime}$. This complites the proof  of Theorem \ref{Inverse_Schur-Weyl}.

 \section{The Schur-Weyl duality for ${\rm Aut}\, M$ and the infinite symmetric group}\label{infinite_Sch_Weyl}
Let $\overline{\mathfrak{S}}_\infty$ be the group of all bijections of the set $\mathbb{Z}_{>0}=\left\{ 1,2,\ldots \right\}$. Set $\mathfrak{S}_n=\left\{s\in \overline{\mathfrak{S}}_\infty: s(k)=k ~\text{for all}~ k>n\right\}$.

Further we will consider $L^2\left( M,{\rm tr} \right)^{\otimes n}$ as the subspace of  $L^2\left( M,{\rm tr} \right)^{\otimes (n+1)}$, using the embedding
\begin{eqnarray*}
L^2\left( M,{\rm tr} \right)^{\otimes n}\ni m_1\otimes\ldots\otimes m_n\mapsto m_1\otimes\ldots\otimes m_n\otimes{\rm I}\in L^2\left( M,{\rm tr} \right)^{\otimes (n+1)}.
\end{eqnarray*}
 Let $L^2\left( M,{\rm tr} \right)^{\otimes \infty}$ be the completion of the pre-Hilbert space $\bigcup\limits_{n=1}^\infty L^2\left( M,{\rm tr} \right)^{\otimes n}$. It is convenient to consider $\bigcup\limits_{n=1}^\infty L^2\left( M,{\rm tr} \right)^{\otimes n}$ as the linear span of the vectors $v_1\otimes \cdots\otimes v_n\otimes{\rm I}\otimes{\rm I}\otimes\cdots$, where $v_j\in M$. At the same time, we will to identify  $L^2\left( M,{\rm tr} \right)^{\otimes n}$ with the closure of the linear span of all vectors $v_1\otimes \cdots\otimes v_n\otimes v_{n+1}\otimes\cdots$, where $v_i={\rm I}$ for all $i>n$.
Then the elements $\theta\in{\rm Aut}\, M$ and  $s\in\overline{\mathfrak{S}}_\infty$ act on  $L^2\left( M,{\rm tr} \right)^{\otimes \infty}$ as follows
\begin{eqnarray*}
\mathfrak{N}^{\otimes \infty}(\theta)\left(v_1\otimes\cdots\otimes v_n\otimes \cdots  \right)=\left(
\mathfrak{N}(\theta)v_1 \right)\otimes\cdots\otimes\left(
\mathfrak{N}(\theta)v_n \right)\otimes \cdots;\\
\;^\infty\!\mathcal{P}(s)\left(v_1\otimes \cdots\otimes v_n\otimes\cdots  \right)=v_{s^{-1}(1)}\otimes
\cdots\otimes v_{s^{-1}(n)}\otimes\cdots  .
\end{eqnarray*}
We now have:
\begin{Th}\label{infiniteSW}
$\left\{\mathfrak{N}^{\otimes\infty}\left({\rm Aut}\, M\right) \right\}'=\left\{\,^\infty\!\mathcal{P}\left( \overline{\mathfrak{S}}_\infty \right) \right\}''$.
\end{Th}
\begin{proof}
Let $(k\;l)$ be a transposition that swaps $k$ and $l$. We denote by $\overline{\mathfrak{S}}_{n,\infty}$ the subgroup $\left\{s\in \overline{\mathfrak{S}}_\infty: s(k)=k ~\text{for all}~k\in\left\{1,2,\ldots,n \right\}\right\}$.

Let us prove that
\begin{eqnarray}\label{coincides_subspaces}
L^2\left( M,{\rm tr} \right)^{\otimes n}=\left\{v\in L^2\left( M,{\rm tr} \right)^{\otimes \infty}: \;^\infty\!\mathcal{P}(s)v=v ~\text{for all }~ s\in\overline{\mathfrak{S}}_{n,\infty}\right\}.
\end{eqnarray}
Fix any $\mathbf{v}\in L^2\left( M,{\rm tr} \right)^{\otimes \infty}$ such that  $\;^\infty\!\mathcal{P}(s)\mathbf{v}=\mathbf{v}$ for all $s\in\overline{\mathfrak{S}}_{n,\infty}$.

Take orthonormal basis $\left\{e_k \right\}_{k=0}^\infty$ in $L^2\left( M,{\rm tr} \right)$, where $e_0={\rm I}$ and $e_k\in M$ for all $k$. Denote by $\mathfrak{K}$ a set of all sequences $\mathfrak{k}=\left\{k_i \right\}_{i=1}^\infty$, $k_i\in\{0,1,\ldots\}$ with the property: there exists same natural $N(\mathfrak{k})$ such, that $k_i=0$ for all $i> N(\mathfrak{k})$.
For convenience, we set  $N(\mathfrak{k})=\min\left\{m: k_i=0~\text{ for all }~ i>m \right\}$.
Then the set $\left\{\mathbf{e}_\mathfrak{k}=e_{k_1}\otimes e_{k_2}\otimes \ldots e_{k_{N(\mathfrak{k})}}\otimes {\rm I}\otimes{\rm I}\otimes\ldots\right\}_{\mathfrak{k}\in\mathfrak{K}}$ is an orhonormal basis in $L^2\left( M,{\rm tr} \right)^{\otimes \infty}$.
Set $$\mathbf{v}=\sum\limits_{\mathfrak{k}\in\mathfrak{K}} c_\mathfrak{k}(\mathbf{v})\mathbf{e}_\mathfrak{k}~\text{ where }~ c_\mathfrak{k}(\mathbf{v})\in\mathbb{C}.$$
To prove (\ref{coincides_subspaces}) it is sufficient to establish that $c_\mathfrak{k}(\mathbf{v})=0$ if  $N(\mathfrak{k})>n$.

Consider an orthogonal projection $O_m$  in  $L^2\left( M,{\rm tr} \right)^{\otimes \infty}$ that is defined as follows
\begin{eqnarray}\label{asym_transposition}
\begin{aligned}
O_m\left(\ldots\otimes e_{k_{m-1}}\otimes e_{k_m}\otimes e_{k_{m+1}}\otimes \ldots e_{k_{N(\mathfrak{k})}}\otimes {\rm I}\otimes{\rm I}\otimes\ldots \right)\\
={\rm tr}(e_{k_m})\left(\ldots\otimes e_{k_{m-1}}\otimes{\rm I}\otimes e_{k_{m+1}}\otimes \ldots e_{k_{N(\mathfrak{k})}}\otimes {\rm I}\otimes{\rm I}\otimes\ldots \right).
\end{aligned}
\end{eqnarray}
It is easily seen that the sequence $\left\{\,^\infty\!\mathcal{P}((m\;l)) \right\}_{l=1}^\infty$ converges in the week operator topology to $O_m=\text{w}-\lim\limits_{l\to\infty}\,^\infty\!\mathcal{P}((m\;l))$. Therefore,
\begin{eqnarray}\label{O_n_in_commutant}
O_m\in\left(\,^\infty\!\mathcal{P}(\overline{\mathfrak{S}}_{\infty})\right)^{''}~\text{ for all }~ m, ~\text{ and }~ O_m\mathbf{v}=\mathbf{v} ~\text{ for all }~ m>n.
\end{eqnarray}
Hence, applying (\ref{asym_transposition}), we have $c_\mathfrak{k}(\mathbf{v})=0$ for all $\mathfrak{k}$ such that $N(\mathfrak{k})>n$. This proves equality (\ref{coincides_subspaces}).

According to (\ref{asym_transposition}), we have that the operator $\mathfrak{P}_{n,N}=O_{n+1}O_{n+2}\cdots O_N$, where $N>n$ is an orthogonal projection. Since $\mathfrak{P}_{n,m}\geq \mathfrak{P}_{n,m+1}$ for all $m>n$, there exists the orthogonal projection $\mathfrak{P}_n=\lim\limits_{m\to\infty}\mathfrak{P}_{n,m}$. By (\ref{O_n_in_commutant}), $\mathfrak{P}_n$ belongs to $\left(\,^\infty\!\mathcal{P}(\overline{\mathfrak{S}}_{n,\infty})\right)^{''}$. Using (\ref{asym_transposition}), we obtain
\begin{eqnarray}\label{n-fixed_projection}
\begin{aligned}
\mathfrak{P}_n\left(v_1\otimes v_2\otimes\otimes\ldots\otimes v_n\otimes v_{n+1}\otimes\ldots \otimes v_j\otimes\ldots \right)\\=\left(\prod\limits_{j=n+1}^\infty {\rm tr}(v_j) \right)\left(v_1\otimes v_2\otimes\ldots\otimes v_n\otimes {\rm I}\otimes\ldots \otimes {\rm I}\otimes\ldots \right).
\end{aligned}
\end{eqnarray}
Therefore, $\mathfrak{P}_n\left( L^2\left( M,{\rm tr} \right)^{\otimes \infty} \right)= L^2\left( M,{\rm tr} \right)^{\otimes n}$.

Take operator $B'\in \left\{\mathfrak{N}^{\otimes\infty}\left({\rm Aut}\, M\right) \right\}'$. Since projection  $\mathfrak{P}_n\in\left(\,^\infty\!\mathcal{P}(\overline{\mathfrak{S}}_{n,\infty})\right)^{''}$ and $~\left(\,^\infty\!\mathcal{P}(\overline{\mathfrak{S}}_{n,\infty})\right)^{''}\subset \left\{\mathfrak{N}^{\otimes\infty}\left({\rm Aut}\, M\right) \right\}'$, then operator $~B'_n=\mathfrak{P}_n\,B'\,\mathfrak{P}_n~$  be\-longs $\left\{\mathfrak{N}^{\otimes\infty}\left({\rm Aut}\, M\right) \right\}'$, too.
 It follows from section \ref{finite_inverse_semigroup} that
 \begin{eqnarray*}
 \mathfrak{P}_n\,\mathfrak{N}^{\otimes\infty}(\theta)\,\mathfrak{P}_n=\mathfrak{N}^{\otimes n}(\theta),\;\;\theta\in{\rm Aut}\, M,\\
  \mathfrak{P}_n \,^\infty\!\mathcal{P}(s)\, \mathfrak{P}_n=\,^n\!\mathcal{P}(s),\, \;\; ~\text{for all }~ s\in \mathfrak{S}_n,\\
  \mathfrak{P}_n\,O_i\, \mathfrak{P}_n =\,^k\!\mathcal{P}^\mathscr{I}(\epsilon_i),\;\; i=1,2,\ldots,n.
 \end{eqnarray*}
 Hence, applying Theorem \ref{Inverse_Schur-Weyl}, we obtain that $B_n'$ belongs to $\left(\,^\infty\!\mathcal{P}(\overline{\mathfrak{S}}_{\infty})\right)^{''}$ (see (\ref{O_n_in_commutant})).
 Since $B'= \lim\limits_{n\to\infty}$ in the strong operator topology, operator $B'$ lies in $\left(\,^\infty\!\mathcal{P}(\overline{\mathfrak{S}}_{\infty})\right)^{''}$, too. This complites the proof of Theorem \ref{infiniteSW}.
\end{proof}
\section{A mapping from unitary to doubly stochastic matrices}\label{U_to_DS}

Recall that $n\times n$-matrix $P=\left[ P_{ij} \right]$ is called {\it doubly stochastic} if $\sum\limits_{i=1}^n P_{ij}=1$, $\sum\limits_{j=1}^n P_{ij}=1$ and $P_{ij}\geq 0$ for all $i,j$. The property of $P$ being doubly stochastic is obviously equivalent to the vector $\left(\begin{smallmatrix}1\\ 1\\ \vdots\\ 1\end{smallmatrix}\right)$ being invariant both for $P$ and the transpose $P^t$. Let $\mathcal{DS}_n$ stand for the set of all doubly stochastic $n\times n$ matrices. There exists an orthogonal matrix $O=\left[ O_{ij} \right]$ such that for any $P\in\mathcal{DS}_n$ one has $\left(OPO^{-1}\right)_{1j}=\delta_{1j}$ and $\left(OPO^{-1}\right)_{j1}=\delta_{j1}$  $(j=1,2,\ldots,n)$, where $\delta_{kl}$ is the Kronecker delta. Let us fix such matrix $O$.

\begin{Lm}Let $\,^1_\gamma\mathbb{M}_n(\mathbb{R})$ be the set of all real $n\times n$ matrices of the form
$\left[ \begin{matrix}\gamma&0&0&\cdots&0\\0&a_{22}&a_{23}&\cdots&a_{2n}\\0&a_{32}&a_{33}&\cdots&a_{3n}\\ \vdots&\vdots&\vdots&\cdots&\vdots
\end{matrix} \right]$.
Suppose that a doubly stochastic matrix $P=\left[ P_{ij} \right]$ has only nonzero entries. Then there exists $\kappa>0$ such that the matrix $P+O^{-1}BO$ is  doubly stochastic for any matrix $B=\left[ B_{ij} \right]\in \,^1_0\mathbb{M}_n(\mathbb{R})$ such that $\left|B_{ij} \right|<\kappa$ for all $i,j$.
\end{Lm}

By the above Lemma, each double stochastic matrix $P$ with positive entries is an interior point of $\mathcal{DS}_n$, and the real dimension of the tangent space $T_P\,\mathcal{DS}_n$ at this point is $(n-1)^2$. In addition, we have a linear one-to-one map between $T_P\,\mathcal{DS}_n$ and $\,^1_0\mathbb{M}_n(\mathbb{R})$.

We need in the sequel the obvious claim as follows.

\begin{Prop}
Let $\mathcal{U}$ be a open subset in $\mathcal{DS}_n$, and $GL(n,\mathbb{R})$ stand for the group of real invertible $n\times n$ matrices. Identify the group $GL(n-1,\mathbb{R})$ with the subgroup $\left(O^{-1}\cdot\,^1_1\mathbb{M}_n(\mathbb{R})\cdot O\right)\cap GL(n,\mathbb{R})\subset GL(n,\mathbb{R})$. Then the topological component of the identity in $GL(n-1,\mathbb{R})$ is contained in $$\bigcup\limits_{j=1}^\infty\left(\left(\mathcal{U}\cap GL(n,\mathbb{R}) \right)\cdot\left(\mathcal{U}\cap GL(n,\mathbb{R})\right)^{-1}\right)^j.$$
\end{Prop}

\subsection{}
Denote by $U(n)$  a group of unitary $n\times n$-matrices. We will consider $U(n)$ and $\mathcal{DS}_n$ as a real manifolds of the dimension $n^2$ and $(n-1)^2$ respectively. Let $f: U(n)\mapsto \mathcal{DS}_n$ be a smooth map and let ${\rm d}f_u$ be a differential of $f$ in the point $u$. Mapping    ${\rm d}f_u$ is the linear operator from the tangent space $T_uU(n)$ at $u$ to the tangent space  $T_{f(u)}\mathcal{DS}_n$.
Function $f$ is a {\it submersion} at a point $u\in U(n)$ if  ${\rm d}f_u \,T_uU(n)=T_{f(u)}\mathcal{DS}_n$.
In connection with formula (\ref{Operator_unistochastic}) we will find the unitary matrices $u$ such that  the map
\begin{eqnarray}\label{mu_def}
U(n)\ni u=\left[ u_{ij} \right]\stackrel{\mu}{\mapsto}\left[\left|u_{ij}\right|^2\right]\in\mathcal{DS}_n \text{ is submersion at the point } u.
\end{eqnarray}
Hence will follow that there exists the open neighborhood $\mathcal{U}$ of the point $u$ such that $\mu(\mathcal{U})\subset\mathcal{DS}_n$ is open subset.

We adopt below the results of A. Karabegov \cite{Karabegov} to make them applicable to proving Proposition \ref{commutants_mutually}.

Denote by $\mathcal{SH}_n$ the set of all skew-hermitian $n\times n$-matrices. It is clear, that the dimension of $U(n)$, as a real manifold, is equal $n^2$. Considering  the smooth one parameter family $U(t)=\left[  U_{kl}(t) \right]\subset U(n)$ and using the equality $U(t)^*\cdot U(t)=\mathrm{I}_n$, we obtain
\begin{eqnarray*}
U(0)^*\cdot U'(0)+U'(0)^*\cdot U(0)=0,\text{\ where\ } U'(0)=\left[U_{kl}'(0)\right].
\end{eqnarray*}
Hence
\begin{eqnarray}\label{skew_hermitian}
 U'(0)\cdot U(0)^*+U(0)\cdot U'(0)^*=0.
\end{eqnarray}

This implies that $U'(0)\in T_uU(n)$ is identified with the skew Hermitian matrix $X=u^*\cdot U'(0)\in T_{\mathrm{I}_n}U(n)$ treated as an element of the Lie algebra $\mathcal{SH}_n$ of $U(n)$. Here $u=\left[u_{kl}\right]=U(0)$.

 Applying (\ref{mu_def}), we see that ${\rm d}\mu_u:T_uU(n)\mapsto T_{\mu(u)}\mathcal{DS}_n$ acts  as follows
\begin{eqnarray*}
{\rm d}\mu_u \;\left( U'(0) \right)=\left[ u_{kl}\overline{U_{kl}^\prime(0)}+ U_{kl}^\prime(0)\overline{u_{kl}}\right]\in T_{\mu(u)}\mathcal{DS}_n.
\end{eqnarray*}
Let us introduce the operator $\,^u\!{\rm d}\mu_u: T_{{\rm I}_n}U(n)\mapsto T_{\mu(u)}\mathcal{DS}_n$ which acts by
\begin{eqnarray}
\,^u\!{\rm d}\mu_u(A)={\rm d}\mu_u(uA),\;A\in T_{{\rm I}_n}U(n),\;uA\in T_uU(n).
\end{eqnarray}
Therefore,
\begin{eqnarray*}
\,^u\!{\rm d}\mu_u\left( u^*U'(0) \right)=\left[ u_{kl}\overline{U_{kl}^\prime(0)}+ U_{kl}^\prime(0)\overline{u_{kl}}\right]\in T_{\mu(u)}\mathcal{DS}_n.
\end{eqnarray*}
Hence, assuming that all entries of $u=U(0)=\left[ u_{kl} \right]$ are nonzero, we obtain
\begin{eqnarray}\label{formula_for_dif}
\,^u\!{\rm d}\mu_u\left( u^*U'(0) \right)=\left[ \left(\frac{U_{kl}^\prime(0)}{u_{kl}}+\frac{\overline{U_{kl}^\prime(0)}}{\overline{u_{kl}}}  \right)\left| u_{kl}\right|^2 \right].
\end{eqnarray}
Now we can to rewrite the equality (\ref{skew_hermitian}) as follows
\begin{eqnarray}\label{coordinate_skew_hermitian}
\sum\limits_{j=1}^n u_{kj}\frac{U_{kj}^\prime(0)}{u_{kj}}\overline{u_{lj}}+\sum\limits_{j=1}^n u_{kj}\frac{\overline{U_{lj}^\prime(0)}}{\overline{u_{lj}}}\overline{u_{lj}}=0.
\end{eqnarray}
Consider the family $\,^\theta\!U=\left[ \,^\theta\!U_{kl} \right]$ of the unitary matrices, where
\begin{eqnarray}\label{Theta_U}
\,^\theta\!U_{kl}=\delta_{kl}+\frac{\theta-1}{n}, \theta\in \mathbb{T}=\left\{ z\in\mathbb{C}:|z|=1\right\}.
\end{eqnarray}

 On the space $\mathbb{M}_n$ of all complex $n\times n$-matrices define two inner ptoducts
\begin{eqnarray*}
&\left<A,B \right>_\theta =\sum\limits_{k,l-1}^n  A_{kl}  \overline{B_{kl}}\left| \,^\theta\!U_{kl}\right|^2,\; A=\left[ A_{kl} \right], B=\left[ B_{kl} \right],\\
&\left<A,B \right>_{\rm Tr}={\rm Tr}\left( AB^* \right), \text{ where }\;{\rm Tr} \;\text{ is an ordinary trace on } \mathbb{M}_n.
\end{eqnarray*}
Denote by $\mathbb{M}_n^\theta$ and $\mathbb{M}_n^{\rm Tr}$ the corresponding  Hilbert spaces.

Now we introduce two operators $\mathbf{C}_\theta$ and $\mathbf{D}_\theta$ as follows
\begin{eqnarray*}
\mathbb{M}_n^\theta\ni f=\left[ f_{kl} \right]\stackrel{\mathbf{C}_\theta}{\mapsto} Y=\left[ Y_{kl} \right]\in \mathbb{M}_n^{\rm Tr}, \text{ where } Y_{kl}=\sum\limits_{j=1}^n \,^\theta\!U_{kj}f_{kj}\overline{\,^\theta\!U_{lj}};\\
\mathbb{M}_n^\theta\ni g=\left[ g_{kl} \right]\stackrel{\mathbf{D}_\theta}{\mapsto} Z=\left[ Z_{kl} \right]\in \mathbb{M}_n^{\rm Tr}, \text{ where } Z_{kl}=\sum\limits_{j=1}^n \,^\theta\!U_{kj}g_{lj}\overline{\,^\theta\!U_{lj}}.
\end{eqnarray*}
Hence, using the orthogonality relations between $\,^\theta\!U_{kj}$, can obtain the formulas for the inverse operators
\begin{eqnarray}\label{inverse_operators}
\left( \mathbf{C}_\theta^{-1}Y \right)_{kq}=\,^\theta\!U_{kq}^{-1}\sum\limits_{j=1}^n Y_{kj}\,^\theta\!U_{jq}\text{ and }
\left( \mathbf{D}_\theta^{-1}Y \right)_{kq}=\,\overline{^\theta\!U}_{kq}^{-1}\sum\limits_{j=1}^nY_{jk} \,\overline{^\theta\!U}_{jq}.
\end{eqnarray}
Set $u=U(0)=\,^\theta\!U$, $X=u^*U'(0)$, $ f_{kj}=\frac{U_{kj}^\prime(0)}{u_{kj}}$ and  $\overline{f}=\left[\,\overline{f_{kj}}\,\right]$.
 Then
 \begin{eqnarray}
 uXu^*= U'(0)\cdot u^* = \mathbf{C}_\theta f \text{ and } uX^*u^*=u\cdot  U'(0)^*=\mathbf{D}_\theta\overline{f}.
 \end{eqnarray}
Hence, applying (\ref{coordinate_skew_hermitian}), we have
 \begin{eqnarray}\label{properties_S_D}
 \mathbf{C}_\theta f=uXu^*, \mathbf{D}_\theta \overline{f}=-uXu^*.
 \end{eqnarray}
It easy to check that the next statement  holds.
\begin{Prop}[Proposition 2.1 \cite{Karabegov}]
If $\theta\notin\{-1,1\}$ then the mappings $\mathbf{C}_\theta$ and $\mathbf{D}_\theta$ are unitary isomorphisms between the Hilbert spaces
$\mathbb{M}_n^\theta$ and $\mathbb{M}_n^{\rm Tr}$.
\end{Prop}
Furthermore, using (\ref{formula_for_dif})
and (\ref{properties_S_D}), we obtain for $X=u^*U'(0)$ and $u=\,^\theta\!U$
\begin{eqnarray}\label{44}
\left(\,^u\!{\rm d}\mu_u \;X\right)_{kl}=\left(\mathbf{C}_\theta^{-1}(uXu^*)-\mathbf{D}_\theta^{-1}(uXu^*)\right)_{kl}\cdot |u_{kl}|^2.
\end{eqnarray}
 Now we will prove the next statement.
 \begin{Th}[Theorem 5.1 \cite{Karabegov}]\label{Karabegov_TH}
 Let $u=\,^\theta\!U$, where $\theta\notin\{-1,1\}$. Then the dimension of the kernel of the operator $\left(\mathbf{C}_\theta^{-1}-\mathbf{D}_\theta^{-1}\right)$ is equal to $2n-1$.
 \end{Th}
 Since the real dimensions of $T_uU(n)$ and $T_{\mu(u)}\mathcal{DS}_n$ are equal $n^2$ and $(n-1)^2$, applying (\ref{44}), we obtain the next
 \begin{Co}\label{Collary_of_Kar}
 If $\theta\notin\{-1,1\}$ then the spaces ${\rm d}\mu_u \left( T_uU(n)\right)$  and $T_{\mu(u)}\mathcal{DS}_n$ coincide.
  \end{Co}
 \begin{proof}[\it Proof of Theorem \ref{Karabegov_TH}. ]
 Let $\mathfrak{D}_n$ be the set of  all diagonal matrices  in $\mathcal{SH}_n$ and let $\mathrm{K}_n$ be a real subspace of $\mathcal{SH}_n$, generated by $\mathfrak{D}_n$ and $u\mathfrak{D}_nu^*$. The ordinary calculations shows that
 \begin{eqnarray}\label{dim_kernel}
 \mathbf{C}_\theta^{-1}\eta =\mathbf{D}_\theta^{-1}\eta \text{ for all } \eta \in \mathrm{K}_n \text{ and }\; {\rm dim}\,\mathrm{K}_n=2n-1.
 \end{eqnarray}
 Define the entries of the matrix $\,^k_l\!B=\left[ \,^k_l\!B_{pq} \right]$ as follows
 \begin{eqnarray}
 \,^k_l\!B_{pq}=\left\{
\begin{array}{rl}
0, &\text{ if }  p=q \text{ or } (p\notin\{k,l\})\bigwedge (q\notin\{k,l\});\\
-1 &\text{ if } p=k, q=l;\\
1, &\text{ if }  p=l, q=k;\\
\frac{n+\overline{\theta}-1}{(\overline{\theta}-1)(n-2)}, &\text{ if }  q=l, p\neq k\text{ and } p\neq l;\\
\frac{n+\theta-1}{(\theta-1)(n-2)}, &\text{ if }  p=k, q\neq l\text{ and } q\neq k;\\
-\frac{n+\overline{\theta}-1}{(\overline{\theta}-1)(n-2)}, &\text{ if }  q=k, p\neq k\text{ and } p\neq l;\\
-\frac{n+\theta-1}{(\theta-1)(n-2)}, &\text{ if }  p=l, q\neq l\text{ and } q\neq k.
\end{array}
\right.
 \end{eqnarray}

 Let $\mathrm{B}_n$ be a real subspace of  $\mathcal{SH}_n$, generated by the matrices $\,^k_l\!B$, where $k,l=1,2,\ldots, n$.
 By the calculations can be can be checked that the subspaces $\mathrm{K}_n$ and $\mathrm{B}_n$  mutually orthogonal and
 \begin{eqnarray}\label{B}
 \mathbf{C}_\theta^{-1}\eta=-\frac{n+\overline{\theta}-1}{n+\theta-1}\mathbf{D}_\theta^{-1}\eta \text{ for all } \eta\in\mathrm{B}_n.
 \end{eqnarray}
 It easy to check that the matrices $\,^1_2\!B$, $\,^2_3\!B$, $\ldots$, $\,_{\;\;\;\;\;\;\,n}^{(n-1)}\!B$ are linearly independent. Therefore,
 \begin{eqnarray}\label{dimension_inequlity}
 {\rm dim}\, \mathrm{B}_n\geq n-1.
 \end{eqnarray}

 Let $\mathrm{O}_n$ be one dimensional subspace $\mathbb{R}\,iO\subset\mathcal{SH}_n$, where $O=\left[  O_{kl} \right]=\left[ \delta_{kl}-1 \right]$. By calculations we see that $\mathrm{K}_n$ and $\mathrm{B}_n$ are orthogonal to $\mathrm{O}_n$ and
 \begin{eqnarray}\label{O}
  \mathbf{C}_\theta^{-1}\,O=-\theta\frac{n+\overline{\theta}-1}{n+\theta-1}\mathbf{D}_\theta^{-1}\,O.
 \end{eqnarray}
 Denote by $\mathrm{IS}_n$ the real subspace of the matrices $A=\left[ A_{kl} \right]\in\mathcal{SH}_n$ with the purely imaginary entries such that
 \begin{eqnarray}\label{sum_equal_zero}
 A_{kk}=0 \text{ and } \sum\limits_{l=1}^n A_{kl}=0 \text{ for all } k=1,2,\ldots, n.
 \end{eqnarray}
 Hence, using (\ref{inverse_operators}), we obtain
 \begin{eqnarray}\label{IS}
 \mathbf{C}_\theta^{-1}\,A=-\overline{\theta}\,\mathbf{D}_\theta^{-1}\,A \text{ for all }A\in\mathrm{IS}_n.
 \end{eqnarray}
 At last we introduce the real subspace $\mathrm{RS}_n$ of the matrices $A=\left[ A_{kl} \right]\in\mathcal{SH}_n$ with the real entries which satisfy  (\ref{sum_equal_zero}). It follows, by the similar calculations, that
 \begin{eqnarray}\label{RS}
  \mathbf{C}_\theta^{-1}\,A=\overline{\theta}\,\mathbf{D}_\theta^{-1}\,A \text{ for all }A\in\mathrm{RS}_n.
 \end{eqnarray}
 Applying (\ref{sum_equal_zero}), we obtain
 \begin{eqnarray}\label{equality_dim_IS}
 {\rm dim}\,\mathrm{IS}_n=\left( \sum\limits_{j=1}^{n-1}(n-j) \right)-n=\frac{n(n-3)}{2}.
 \end{eqnarray}
 Analogously,
 \begin{eqnarray}\label{equality_dim_RS}
 {\rm dim}\,\mathrm{RS}_n=\left( \sum\limits_{j=1}^{n-1}(n-j) \right)-(n-1)=\frac{(n-1)(n-2)}{2}.
 \end{eqnarray}
 By the ordinary calculations  can to show that subspaces $\mathrm{K}_n$, $\mathrm{B}_n$, $\mathrm{O}_n$, $\mathrm{IS}_n$, $\mathrm{RS}_n$ are pairwise orthogonal. Hence, applying (\ref{dim_kernel}), (\ref{dimension_inequlity}), (\ref{equality_dim_IS}) and (\ref{equality_dim_RS}), we have
$$ {\rm dim}\left( \mathrm{K}_n\oplus\mathrm{B}_n \oplus\mathrm{O}_n\oplus\mathrm{B}_n\oplus\mathrm{IS}_n\oplus\mathrm{IR}_n\right)\geq n^2.$$
Therefore, $\mathrm{K}_n\oplus\mathrm{B}_n \oplus\mathrm{O}_n\oplus\mathrm{B}_n\oplus\mathrm{IS}_n\oplus\mathrm{IR}_n=\mathcal{SH}_n$.
Thus any $\Psi\in\mathcal{SH}_n$ can to write as follows $\Psi=\Psi_K+\Psi_B+\Psi_O+\Psi_{IS}+\Psi_{RS}$, where $\Psi_{\star}$ lies in the corresponding orthogonal component. If $\Psi$ lies in kernel of the operator ${\rm d}\mu_u=\left(\mathbf{C}_\theta^{-1}-\mathbf{D}_\theta^{-1}\right)$  then,  using (\ref{dim_kernel}), (\ref{B}), (\ref{O}), (\ref{IS}) and (\ref{RS}), we obtain
\begin{eqnarray*}
&D_\theta\circ{\rm d}\mu_u\,\Psi=\left( -\frac{n+\overline{\theta}-1}{n+\theta-1}-1+ \right)\Psi_B+\left( -\theta\frac{n+\overline{\theta}-1}{n+\theta-1}-1+ \right)\Psi_O\\
&-(\theta+1)\Psi_{IS}+(\overline{\theta}-1)\Psi_{RS}.
\end{eqnarray*}
Since $\theta\notin\{-1,1\}$, then $\Psi_B=\Psi_O=\Psi_{IS}=\Psi_{RS}=0$. Therefore, $\Psi=\Psi_K\in\mathrm{K}_n$.
 \end{proof}
 The next statement follows from Corollary  \ref{Collary_of_Kar}.
 \begin{Co}\label{submersion}
 If  $\theta\notin \{-1,1\}$ then ${\rm d}\mu_u$  is submersion at the point $u=\,^\theta\!U$. Therefore,  there exists an open subset $\mathcal{U}$ such that $u\in\mathcal{U}$ and $\mu(\mathcal{U})$ is  an open subset in $\mathcal{DS}_n$.
 \end{Co}

{}
B. Verkin ILTPE of NASU - B.Verkin Institute for Low Temperature Physics and Engineering
of the National Academy of Sciences of Ukraine

n.nessonov@gmail.com


\begin{thebibliography}{}

\bibitem{APat}
Beltita D., Neeb KH.,
Schur-Weyl Theory for C*-algebras, Mathematische Nachrichten 285 (2012), p. 1170 - 1198.
\bibitem{Birkhoff}
D. Birkhoff, Tres observaciones sobre el algebra lineal. Univ. Nac.Tucuman
Rev, A5 (1946) 147-151

\bibitem{Fulton-Harris}
W. Fulton, J. Harris, Representations theory ((A first  Course), Springer, 1991, 551pp.

\bibitem{NN}
Nessonov N.I.,
An analogue of Schur–Weyl duality for the unitary group of a $II_1$-factor, Mat. Sb., 2019, Volume 210, Number 3, Pages 162–188.

\bibitem{Munn_1}
W. D. Munn, {\it Matrix representations of semigroups}, Proc. Cambridge Philos. Soc., {\bf 51}, 1955,
1-15.
\bibitem{Munn_2}
W. D. Munn, {\it The characters of the symmetric inverse semigroup}, Proc. Cambridge Philos. Soc., {\bf
53}, 1957, 13-18.
\bibitem{Kir}
A. A. Kirillov, Elements of the theory of representations, 2nd ed., Nauka, Moscow
1978, 343 pp.; English transl. of 1st ed., Grundlehren Math. Wiss., vol. 220,
Springer-Verlag, Berlin–New York 1976, xi+315 pp.
\bibitem{Olkin}
A. W. Marshall and I. Olkin, Inequalities: Theory of Majorization and Its
Applications, (Academic Press, New York, 1979), Chapter 2


\bibitem{Grood}
C. Grood, {\it  A Specht Module Analog for the Rook monoid}, The Electronic Journal of Combinatorics 9
(2002), $\#R2$.

\bibitem{East}
East J., Generators and relations for partition monoids and algebras, Journal of Algebra 339 (2011) 1–26

\bibitem{Popova}
 Popova L.M., Defining relations in some semigroups of partial transformations of a finite set, Uchenye Zap. Leningrad Gos. Ped. Inst. 218 (1961) 191–212 (in Russian).

\bibitem{Karabegov}
 Karabegov A., A mapping from the unitary to double stochastic matrices and symbol on a finite set;arXiv: 0806.2357v1 [math. OA]  14Jun 2008, 13 pp.
\bibitem{Ner}
Neretin Yu. A.,
{\it Categories of bistochastic measures, and representations of some infinite-dimensional groups},
Russian Acad. Sci. Sb. Math., 75:1 (1993), 197 - 219
\bibitem{TAKES1}
Takesaki M., {\it Theory of Operator Algebras, v. ${\rm I}$},
 Springer, 2005, 416  pages.

\bibitem{Sinclair}
Sinclair A., Smith R., {\it Finite von Neumann Algebras and Masas}, Cambridge University Press, London
Mathematical Society, Lecture Notes  Series, 351, 2008, 400 pages.
\end{thebibliography}
\end{document}